\documentclass[11pt,a4paper]{amsart}

\usepackage{amsfonts}
\usepackage{amsmath,amscd,amsthm}
\usepackage{fullpage}
\usepackage{amssymb}
\usepackage{centernot} 
\usepackage{enumerate} 
\usepackage{pb-diagram} 
\usepackage{mathrsfs}
\usepackage[OT2,T1]{fontenc}
\usepackage{seqsplit}
\usepackage{tikz}
\usepackage{tikz-cd}
\usepackage{color}
\usepackage[autostyle]{csquotes}
\usepackage{array}
\usepackage{verbatim}
\usepackage{url}
\usepackage{multirow}
\usepackage[scr=boondoxo]{mathalpha}
\usepackage{colortbl}
\usepackage{graphicx}
\usepackage{adjustbox}

\definecolor{pinegreen}{HTML}{16392C}

\DeclareSymbolFont{cyrletters}{OT2}{wncyr}{m}{n}
\DeclareMathSymbol{\Sha}{\mathalpha}{cyrletters}{"58}

\definecolor{refkey}{rgb}{1,1,1}
\definecolor{labelkey}{rgb}{1,1,1}
\definecolor{cite}{rgb}{0.9451,0.2706,0.4941}
\definecolor{ruri}{rgb}{0.0078,0.4022,0.8010}

\usepackage[%
bookmarks=true,bookmarksnumbered=true,%
colorlinks=true,linkcolor=ruri,citecolor=red%
 ]{hyperref}

\makeindex 

\def\F{{\rm \mathbb{F}}}
\def\Z{{\rm \mathbb{Z}}}

\def\Q{{\rm \mathbb{Q}}}

\def\G{{\rm \mathbb{G}}}

\def\C{{\rm \mathbb{C}}}

\def\R{{\rm \mathbb{R}}}

\def\P{{\rm \mathbb{P}}}

\def\p{{\rm \mathfrak{p}}}
\def\m{{\rm \mathfrak{m}}}

\def\O{{\rm \mathcal{O}}}

\def\A{{\rm \mathbb{A}}}

\def\Nm{{\rm Nm}}

\def\un{{\rm un}}

\def\avg{{\rm avg}}

\def\inv{{\rm inv}}
\def\Aut{{\rm Aut}}

\def\Tr{{\rm Tr}}

\def\Br{{\rm Br}}

\def\Cl{{\rm Cl}}

\def\Diff{{\rm Diff}}
\def\Disc{{\rm Disc}}
\def\disc{{\rm disc}}

\def\SL{{\rm SL}}
\def\SO{{\rm SO}}

\def\Res{{\rm Res}}
\def\Sym{{\rm Sym}}

\def\GL{{\rm GL}}

\def\Gal{{\rm Gal}}

\def\rk{{\rm rk}}

\def\Hom{{\rm Hom}}

\def\Sel{{\rm Sel}}

\def\coker{{\rm coker \hspace{1mm}}}

\numberwithin{equation}{section}

\newtheorem{theorem}{Theorem}[section]
\newtheorem{lemma}[theorem]{Lemma}

\newtheorem{remark}[theorem]{Remark}
\newtheorem{definition}[theorem]{Definition}
\newtheorem{example}[theorem]{Example}
\newtheorem{conjecture}[theorem]{Conjecture}
\newtheorem{corollary}[theorem]{Corollary}
\newtheorem{proposition}[theorem]{Proposition}

\begin{document}
\title{Hecke reciprocity and class groups}

\author{Ari Shnidman}
\address{Institute for Advanced Study, 1 Einstein Drive, Princeton NJ 08540}
\email{ari.shnidman@gmail.com}
\author{Artane Siad}
\address{Institute for Advanced Study, 1 Einstein Drive, Princeton NJ 08540}
\email{ajsiad@ias.edu}
\date{\today}
\makeatletter
\DeclareRobustCommand{\pmods}[1]{\mkern4mu({\operator@font mod}\mkern6mu#1)}
\makeatother

\begin{abstract}
We compute the average size of $\Cl_F[2]$ in the family of cubic fields $F = \Q(\sqrt[3]{n})$. Specifically, as $F$ varies over the subfamily of wildly (resp.\ tamely) ramified fields $\Q(\sqrt[3]{n})$, the average size of $\Cl_F[2]$ is $3/2$ (resp.\ $2$). This tame/wild dichotomy is not accounted for by the class group heuristics in the literature.  Analogously, when the extensions $F = K(\sqrt[3]{n})$ of $K = \Q(\sqrt{-3})$ are ordered by the norm of $n \in \O_K$, we show that the average size of $\Cl_F[2]$ is $3/2$, as is predicted by the Cohen--Martinet heuristics for $C_3$-extensions of $K$.   

Underlying our proofs is a reciprocity law for the relative class groups $\Cl_{F/K}[2]$ of odd degree extensions of number fields $F/K$. 
% This leads to a random model for $\Cl_{F/K}[2]$ that recovers Malle's heuristics for $\mathrm{Cl}_{F/K}[2]$ in families of $K$-extensions with fixed Galois group.
This leads us to propose class group heuristics for families of $K$-extensions with a fixed Galois $K$-group that explains the aberrant  behavior in the family $\Q(\sqrt[3]{n})$ and predicts similar behavior in other special families. The other main ingredient is the work of Alp\"oge--Bhargava--Shnidman on the number of integral $G(\Q)$-orbits in a $G$-invariant quadric with bounded invariants.  
\end{abstract}
\maketitle

\tableofcontents

\makeatletter
\makeatother

\section{Introduction}
    Given a subgroup $G \subset S_m$, a prime $\ell \nmid m$, and a number field $K$, the Cohen--Lenstra--Martinet heuristics predict the distribution of the $\ell$-part of the relative class group $\Cl_{F/K}[\ell^\infty]$, as $F$ varies over $K$-extensions of degree $m$ with Galois group $G$ \cite{CohenLenstra1983,CohenMartinet90}. Despite extensive evidence for these heuristics in the function field setting \cite{EllenbergVenkateshWesterland,LiuWoodZureickBrown,LandesmanLevy}, little is known over number fields beyond the seminal works of Davenport--Heilbronn \cite{DavenportHeilbronn} and Bhargava \cite{BhargavaQuartics} and their generalizations \cite{DatskovskyWright,BhargavaShankarWangGlobalFields,LemkeOliverWangWood}.

Our first result confirms a prediction of the Cohen--Martinet heuristics for $G = C_3$ and $K = \Q(\sqrt{-3})$.  Note that each $C_3$-extension of $K$ is of the form $F_n = K(\sqrt[3]{n})$, for some $n \in \O_K$. 

\begin{theorem}\label{thm: over Q(omega)}
    Let $K = \Q(\sqrt{-3})$ and let $N(X) = \{n \in \O_K \mbox{ cubefree}: 1 < \Nm(n) < X\}$. Then 
    \[\displaystyle\lim_{X \to \infty}\frac{1}{|N(X)|} \sum_{n \in N(X)} |\Cl_{F_n}[2]| = 3/2.\]
\end{theorem}

Our second result is the analogous result over $\Q$, for the family of pure cubic fields $F = \Q(\sqrt[3]{n})$.  These are non-Galois, with Galois group $S_3$, so we are in the subtle case where $\ell \nmid [F \colon \Q]$ but $\ell \mid |G|$. Here we observe phenomena that is unaccounted for by the heuristics in the literature. 
\begin{theorem}\label{main thm: over Q} Let $W(X)$ $($resp.\ $T(X))$ be the set of cubefree integers $1 < n < X$ congruent $($resp.\ not congruent$)$ to $\pm1\pmod{9}$, and let $F_n = \Q(\sqrt[3]{n})$. Then

    \begin{enumerate}
        \item \[\lim_{X \to \infty}\frac{1}{|W(X)|} \sum_{n \in W(X)} |\Cl_{F_n}[2]| = 2;\]
        \item \[\lim_{X \to \infty}\frac{1}{|T(X)|} \sum_{n \in T(X)} |\Cl_{F_n}[2]| = 3/2.\]
    \end{enumerate}
In particular, for at least $50\%$ of cubefree $n \equiv \pm 1\pmod9$, the class number of $\Q(\sqrt[3]{n})$ is odd.
\end{theorem}

Theorem \ref{main thm: over Q}(2) is in accordance with the Cohen--Martinet--Malle heuristics \cite{Malle2010} for the larger family of complex $S_3$-cubic fields.\footnote{Their heuristics were originally stated with respect to the discriminant ordering, whereas we use an ordering that is specific to Kummer extensions.  Nowadays, one expects these heuristics to hold for most reasonable orderings, where what is  ``reasonable'' may depend on the family; indeed, the Cohen--Lenstra heuristics are known to not always hold in the discriminant ordering \cite{BartelJohnstonLenstra}. 
}
However, Theorem \ref{main thm: over Q}(1) shows that the groups $\Cl_{F_n}[2]$ are governed by a {\it different} distribution.  This phenomenon was observed empirically by many of the pioneers of class group statistics. For example, Cohen--Martinet \cite{CohenMartinet87} write:
\blockquote{\em However, there is another phenomenon which has been stressed several times
 (\cite{ShanksWilliams}, \cite{TennenhouseWilliams}) and which we repeat here: If one considers only $\Q(\sqrt[3]{p})$ with $p\equiv 2\pmod 3$ prime (so as not to be bothered by the $3$-part), and if one distinguishes
 between $p\equiv -1 \pmod9$ and $p\equiv 2, 5 \pmod 9$, one notes a marked distinction in the
 behavior of the class group. For example, class number 1 seems to occur with
 probability $0.60$ for $p \equiv -1 \pmod9$, but with probability $0.40$ for $p \equiv 2, 5 \pmod9$. This is apparently due to the higher 2-part of the class group in the second case, and
 although a sort of reinterpretation of this phenomenon has been given in \cite{EisenbeisFreyOmmerbon}, no
 satisfactory heuristic explanation has yet been found.}

Our proof of Theorem \ref{main thm: over Q} explains this mysterious phenomenon and leads us to predict similar behavior in other families of extensions $F/K$.  The new insight is that the group $\Cl_{F/K}[2]$ is not as random as it seems: it obeys a global reciprocity law that we call \textit{Hecke reciprocity}.

\subsection{Hecke reciprocity}

Let $F/K$ be an odd degree extension of number fields. Class field theory gives a bijection from $\Cl_{F/K}^\vee[2] = \Hom(\Cl_{F/K},\Z/2\Z)$ to the set of unramified extensions $F(\sqrt{t})/F$ of square norm, i.e.\ such that $\Nm_{F/K}(t) \in K^{\times2}$. Roughly speaking, Hecke reciprocity constrains the number of ramified primes that are inert in $F(\sqrt{t})/F$.  

To state it precisely, let $\Disc_{F/K}\subset \O_K$ and $\Diff_{F/K} \subset \O_F$ be the discriminant and different ideals of $F/K$, and define the {\it Hecke ideal} 
 \[H_{F/K} := \Disc_{F/K} \mathrm{Diff}_{F/K}^{-1} \subset \O_F.\] As $F/K$ has odd degree, the norm of $H_{F/K}$ is a square (as an $\O_K$-ideal). Moreover,  
 a well-known theorem of Hecke \cite[pg. 291]{WeilBasicNumberTheory}, says that $[H_{F/K}]$ is a square in $\Cl_{F/K}$.  

 \begin{definition}
 {\em 
     We say $F/K$ is {\it Hecke ramified} if $H_{F/K}$ is {\it not} the square of an $\O_F$-ideal.
     }
 \end{definition}
 
 Write $H_{F/K} = \prod_w w^{h_w}$ for the decomposition of $H_{F/K}$ into prime powers.  
 \begin{definition}
  {\em   A prime $w$ of $F$, lying over a prime $v$ of $K$, is a {\it Hecke prime} if $h_w$ is odd and if, locally, $w$ is inert in some unramified quadratic extension of $F \otimes_K K_v$ of square norm.  
    }  
 \end{definition}

 Thus, a Hecke prime $w$ is a witness to the fact that $F/K$ is Hecke ramified.  The secondary condition rules out primes $w$ that split in all extensions $F(\sqrt{t}) \in \Cl_{F/K}^\vee[2]$ for local reasons. One checks that $F/K$ is Hecke ramified if and only if $F/K$ admits a Hecke prime (Lemma \ref{lem: hecke ramified implies hecke prime}).   
 Moreover, if $\Disc_{F/K}$ has valuation $1$ at some prime $v$ of $K$, then there is a unique Hecke prime $w$ above $v$. This shows that a ``random'' odd degree extension $F/K$ has many Hecke primes. 
 
With this definition in hand, the Hecke reciprocity law states:

\begin{theorem}[Hecke reciprocity]\label{thm: reciprocity law for odd degree intro}
    Let $F/K$ be an odd degree extension of number fields and let $F(\sqrt{t})/F$ be an everywhere unramified quadratic extension of square norm. Then the number of Hecke primes $w$ of $F$ that are inert in $F(\sqrt{t})$ is even.  
\end{theorem}

Theorem \ref{thm: reciprocity law for odd degree intro} follows quickly from Hecke's theorem (see \S\ref{subsec: Hecke reciprocity}). Its relevance to class group heuristics is that it causes ``unexpected congruence conditions'' on elements of $\Cl_{F/K}^\vee[2]$. For example, if $F$ is a cubic field in which a prime $p$ splits as  $(p) = \p_1^2\p_2$ and $\p_1$ is the only Hecke  prime for $F/\Q$, then Theorem \ref{thm: reciprocity law for odd degree intro} implies that for all $F(\sqrt{t}) \in \Cl_F^\vee[2]$, the prime $\p_1$ splits in $F(\sqrt{t})$,  even though there is no good local reason for this! 
This explains the dichotomy in Theorem \ref{main thm: over Q}: when $n \not\equiv \pm1 \pmod{9}$, the field $F_n$ is wildly ramified at $3$ and has no Hecke primes, whereas for $n \equiv \pm 1\pmod{9}$ the field $F_n$ is tamely ramified and admits exactly one Hecke prime, the prime $\p_1$ above $(3) = \p_1^2\p_2$. In the tame case, Hecke reciprocity systematically obstructs the non-trivial unramified local square-class $t$ at $\p_1$ from appearing in a global unramified square-class. Since there are two unramified square-classes at $\p_1$, we expect half as many non-zero classes in $\Cl_F^\vee[2]$ in the tame family compared to the wild family, which indeed follows from Theorem \ref{main thm: over Q}.

\subsection{Methods}

We view $\Cl_{F/K}^\vee[2]$ as a Selmer group for the $\Gal_K$-module
\[M = \ker\left(\Res^F_K \mu_2 \stackrel{\Nm}{\longrightarrow} \mu_2\right)\]
of order $2^{[F \colon K]-1}$, where $\Res^F_K \mu_2$ denotes Weil restriction.  More precisely, class field theory gives an isomorphism from $\Cl_{F/K}^\vee[2]$ to the {\it unramified $2$-Selmer group} of $F/K$, which is the subgroup
\[\Sel^\mathrm{un}_2(F/K) \subset H^1(K, M) \subset F^\times/F^{\times2}\]
of everywhere unramified classes $t \in H^1(K, M)$ (see Theorem \ref{thm: canonical isomorphism}).
% Class field theory gives an isomorphism $\Cl_{F/K}^\vee[2] \simeq \Sel^\mathrm{un}_2(F/K)$ (Theorem \ref{thm: canonical isomorphism}), which 
This endows $\Cl_{F/K}^\vee[2]$ with a symmetric bilinear pairing coming from the cup product
\[(\, , \, )\colon H^1(K, M) \times H^1(K, M) \longrightarrow H^2(K, \mu_2)[2] = \Br(K)[2].\]
Key to our proof is the notion of a {\it quadratic refinement} of $(\, , \, )$, which is by definition a quadratic map  $q \colon H^1(K,M) \to \Br(K)[2]$ such that $(x,y) = q(x + y) - q(x) - q(y)$.

When $F/K$ is cubic, we show that there exists a quadratic refinement $q = \sum_v q_v$ with the following property:  for any Hecke prime $w \mid v$ of $F$, and for any non-trivial $t \in H^1_{\un}(K_v, M)$, we have $\inv_v(q_v(t)) = 1/2$, where $\inv_v$ is the isomorphism  $\Br(K_v)[2] \simeq \Z/2\Z$.  It follows from reciprocity in the Brauer group that for each $t \in \Sel_2^\un(F/K)$, the number of Hecke primes inert in $F(\sqrt{t})$ is even. This gives a purely algebraic proof of Theorem \ref{thm: reciprocity law for odd degree intro} (and Hecke's theorem as well) for cubic extensions. It is this connection with Brauer groups that is needed in our proof of Theorems \ref{thm: over Q(omega)} and \ref{main thm: over Q}.
% \footnote{One might hope to generalize this argument to give a purely algebraic proof of Hecke's theorem in any degree, but we do not pursue this here.}  
Namely, we exploit the fact that for Kummer extensions $F_n = K(\sqrt[m]{n})$, there is a natural choice of quadratic refinement: the map $q^n \colon H^1(K, M_n) \to \Br(K)[2]$ sending $[t]$ to the class of the quadratic form $\frac{1}{m}\Tr_{F/K}(t x^2)$ in $H^1(K, \mathrm{SO}(m)) \simeq \Br(K)[2]$. 

When $m = 3$, we {\it attempt} to parameterize the elements of $\Sel_2^\un(F_n/K)$, as follows. The space $V = \Sym^3(K^2) \otimes K^2$ of pairs of binary cubic forms has a natural action by $G = \SL_2^2/\mu_2$. The ring $K[V]^G$ is generated by two invariants $A_1$ and $A_3$, and the Galois module $M_n$ is isomorphic to the stabilizer $\mathrm{Stab}_G(v_n)$ of any vector $v_n$ with invariants $A_1 = 0$ and $A_3 = n$. By arithmetic invariant theory, the $G(K)$-orbits with these invariants are in bijection with the kernel of the map $ H^1(K, \mathrm{Stab}_G(v_n)) \to H^1(K,G) \simeq \Br(K)[2]$. An elegant geometric argument of Bhargava-Gross \cite{BhargavaGross2013} shows that the latter map may be identified with $q^n = \sum_v q^n_v$. We then check by explicit computation that $H^1_\un(K_v, M_n)\not\subset\ker(q^n_v)$ if and only if $F_n/K$ is Hecke ramified at $v$ if and only if $3\mid v$ and $F_n$ is tamely ramified at $v$.  If $F_n/K$ is Hecke unramified, this gives a robust parameterization since all elements of $\Sel_2^\un(F_n/K)$ correspond to orbits of $V(K)$. However, if $F_n/K$ is Hecke ramified at two or more primes $v$ of $K$, this parameterization does not work, as elements of $\Sel_2^\un(F_n)$ need not lie in $\ker(q^n)$, hence need not correspond to an orbit of $V(K)$.  Luckily, if $K$ has  only  one prime above $3$, we obtain a parameterization of all Selmer elements $\Sel_2^\un(F_n/K)$, for all $n \in K^\times$. For example, when $K = \Q$ each $t \in \Sel_2^\un(F_n)$ lies in $\ker(q^n_3)$, even if $F_n/\Q$ is tamely ramified at $3$, by Hecke reciprocity.  Hence  $t \in \ker(q^n)$ and  $t$ corresponds to a $G(K)$-orbit on $V(K)$.       

The orbits corresponding to $\Sel_2^\un(F_n/\Q)$ have integral representatives. Thus, to finish the proof of Theorem \ref{main thm: over Q}, we estimate the number of integral orbits lying on the quadric $A_1 = 0$ and with $A_3$-invariant bounded by $X$, using a counting result of Alp\"oge, Bhargava, and the first author in  \cite{AlpogeBhargavaShnidman}.    This last step should work over any number field over which our parameterization described above is available. However, we restrict ourselves to the cases $K = \Q$ and $K$ an imaginary quadratic field of class number $1$, as the arguments in \cite{AlpogeBhargavaShnidman} generalize straightforwardly in the latter case.  We note that it would be more natural to use the five dimensional representation of $\SO(3)$, as in \cite{BhargavaGross2013} and \cite{AlpogeThesis}, but we use $G = \SL_2^2/\mu_2$ in order to invoke  \cite{AlpogeBhargavaShnidman}, which handles the details of the circle method argument and in the generality needed. 

In the wild family $n \not\equiv \pm 1\pmod{9}$,  applying the counting result of \cite{AlpogeBhargavaShnidman} yields the formula 
\[\avg_n |\Cl_{F_n}[2]| = \avg_n |\Sel_2^\un(F_n)| = 1 + \tau(G)\cdot \prod_{p \leq \infty} \nu_p = 1 + 2 \cdot \frac12 \prod_{p < \infty}1 = 1 + 1 = 2,\] 
where  $\nu_p$ is the (weighted) average number of local unramified classes and $\tau(G)$ is the Tamagawa number. In the tame family $n \equiv \pm1 \pmod{9}$, the non-trivial class $t \in H^1_\un(\Q_3,M_n)$ is ``obstructed'' by Hecke reciprocity, aligning with the fact that it does not correspond to a local orbit in $V(\Q_3)$. This gives $\nu_3 = \frac12$ and hence the global average is $1 + 1/2 = 3/2$. Over $K = \Q(\sqrt{-3})$ this dichotomy disappears because odd degree Galois extensions are Hecke unramified (Proposition \ref{prop: G odd}). 

It is interesting to compare our proof with Bhargava's proof that $\mathrm{avg}_F |\Cl_F[2]| = 3/2$ in the family of all complex $S_3$-cubic fields \cite{BhargavaQuartics}. His parametrization uses  $G' =\SL_2 \times \GL_3$ which has $\tau(G') = 1$. The factor $\tau(G) = 2$ in our proof is, in effect, canceled out by Hecke reciprocity, so that both proofs arrive at the same answer, at least in the family of tame pure cubic fields.

\subsection{Heuristics for $G$-extensions}

In  \cite{Malle}, Malle points out that the Cohen--Lenstra--Martinet heuristics should be modified when $K$ contains the $\ell$-th roots of unity $\mu_\ell$; see also \cite{Achter06, Garton2015, LipnowskiSawinTsimerman}. Sawin and Wood \cite{SawinWood} recently proposed the following heuristic in that case, assuming $\ell \nmid |G|$ as before. Let $S = \Z_\ell[G]/\sum_{\gamma \in G} \gamma$, let $r$ be maximal such that $\mu_{\ell^r} \subset K$, and let $u$ be the number of infinite places of $K$.  Sawin and Wood predict that for any finite $S$-module $V$, we have:
\begin{equation} \label{eqn: Sawin--Wood}
    \mathbb{E}_{F \in \mathcal{F}_{G,K}}|\mathrm{Surj}(\Cl_{F/K}[\ell^\infty],V)| = \dfrac{| (\wedge^2 V)^G[\ell^r]|}{|V|^u},
\end{equation}
for any ``reasonable'' ordering on the family $\mathcal{F}_{G,K}$ of $G$-extensions of $K$.\footnote{This heuristic implies heuristics for families of degree $m$ non-Galois fields $F/K$ with fixed Galois group $G_{F/K} \subset S_m$, so as long as $\ell \nmid |G_{F/K}|$, it is enough to consider the Galois case.}  Theorem \ref{thm: over Q(omega)} proves $(\ref{eqn: Sawin--Wood})$ for $(G,K,V) = (C_3,\Q(\sqrt{-3}),\F_4)$. It appears to be the first proven instance of $(\ref{eqn: Sawin--Wood})$ aside from the work of Datskovsky--Wright \cite{DatskovskyWright} for $(G,V) = (C_2,\Z/3\Z)$, generalizing Davenport--Heilbronn's \cite{DavenportHeilbronn}; see also \cite{LemkeOliverWangWood}.

Theorem \ref{main thm: over Q} concerns the non-Galois case where $\ell \mid |G_{F/K}|$ but $\ell \nmid [F \colon K]$.  The heuristics are murkier in this case, but Malle has made conjectures for $\Cl_{F/K}[2]$ in certain families of odd degree extensions with fixed Galois group $G \subset S_m$ of even order \cite{Malle2010}.  For example, let $m \geq 3$ be odd and consider the family $\mathcal{F}$ of $S_m$-extensions $F/\Q$ of degree $m$ and unit rank $u = r_1 + r_2 - 1$.  Malle predicts that for every finite $\F_2$-vector space $V$, 
\begin{equation} \label{eqn: Malle moments}
    \mathbb{E}_{F \in \mathcal{F}}(|\mathrm{Surj}(\Cl_{F/\Q}[2],V)|) = \dfrac{| \wedge^2 V|}{|V|^u},
\end{equation}
In \S\ref{sec: random model for Cl[2]}, we use Hecke reciprocity to give a new model for $\Cl_{F/\Q}[2]$ as a cokernel of an alternating matrix over $\F_2$ modulo $u + 1$ random relations.  This recovers Malle's distribution, albeit via different reasoning (and without appealing to the function field case).  The role of Hecke reciprocity here amounts to adding an additional relation.  This applies to other families $F/K$ with fixed Galois group $G \subset S_m$ and signature, as long as the typical member of the family is Hecke ramified.

\subsection{Heuristics for $\Gamma$-extensions}\label{subsec: Gamma extension heuristics intro}
On the other hand, Theorem \ref{main thm: over Q} shows that  certain natural families of odd degree extensions $F/K$ do {\it not} fit into Malle's model for $\Cl_{F/K}[2]$. To account for this, we generalize Malle's heuristics by fixing not just the Galois group $G_{F/K} \subset S_m$, but part of the natural $\Gal_K$-action on $G_{F/K}$ as well; here $\Gal_K = \Gal(\overline{K}/K)$.

Let $\Gamma$ be a finite $K$-group, i.e.\ $\Gamma$ is a finite group $\Gamma_0$ endowed with an action of $\Gal_K$.
\begin{definition} 
{\em 
A {\it $\Gamma$-extension} is an extension $F/K$ together with an embedding $\Gamma \hookrightarrow G_{F/K}$ of $K$-groups.  We denote by $\mathcal{F}_\Gamma$ the family of $\Gamma$-extensions $F/K$ of degree $|\Gamma|$.
}
\end{definition}
 If $\Gamma$ has trivial $\Gal_K$-action, then $\mathcal{F}_\Gamma$ is simply the family $\mathcal{F}_{\Gamma_0, K}$ of  $\Gamma_0$-extensions. In general, we define the {\it resolvent field} of $\Gamma$ to be the Galois extension $R/K$ trivializing the $\Gal_K$-action on $\Gamma$. We assume for simplicity that $m:= |\Gamma|$ is odd and coprime to $[R \colon K]$. In this case, we have $G_{F/K} \simeq \Gamma \rtimes H$, where $H := G_{R/K}$.  It follows that $\mathcal{F}_\Gamma$ is the family of degree $m$ extensions with Galois group $G_\Gamma := \Gamma_0 \rtimes H$ and resolvent field $R$. For example, if $K= \Q$ and $\Gamma = \mu_3$, then $R = \Q(\mu_3)$, $G_\Gamma = S_3$, and $\mathcal{F}_\Gamma$ is the family $\Q(\sqrt[3]{n})$ appearing in Theorem \ref{main thm: over Q}. 

When $|G_\Gamma|$ is odd, all fields in $\mathcal{F}_{G_\Gamma,K}$  are Hecke unramified (Proposition \ref{prop: G odd}), so Hecke reciprocity is not relevant. In this case, we simply predict that the distribution of $\Cl_{F/K}[2]$ for $F \in \mathcal{F}_\Gamma$ is the one predicted by Sawin--Wood in the larger family $\mathcal{F}_{G_\Gamma,K}$.

 Our more novel prediction is when $|G_\Gamma|$ is even.  For simplicity, we assume here that $G_{F/K} \subset S_m$ is absolutely irreducible in the sense of \cite{SawinWood}. Write $\mathcal{F}_\Gamma = \mathcal{F}_\Gamma^1 \coprod \mathcal{F}_\Gamma^2$, where $\mathcal{F}_\Gamma^1$ (resp.\ $\mathcal{F}_\Gamma^2$) is the subfamily of $\Gamma$-extensions $F/K$ that are Hecke unramified (resp.\ Hecke ramified). Then we predict: \begin{equation} \label{eqn: altered Malle}
    \mathbb{E}_{F \in \mathcal{F}_\Gamma^1}(|\mathrm{Surj}(\Cl_{F/K}[2],V)|) = \dfrac{| \wedge^2 V|}{|V|^{u-1}};
\end{equation}
\begin{equation} \label{eqn: Malle again}
    \mathbb{E}_{F \in \mathcal{F}_\Gamma^2}(|\mathrm{Surj}(\Cl_{F/K}[2],V)|) = \dfrac{| \wedge^2 V|}{|V|^{u}},
\end{equation}
accounting for the vacuity of Hecke reciprocity on $\mathcal{F}_\Gamma^1$. The parameter $u$ is the one defined by Cohen--Martinet, as predicted by Malle in the larger family $\mathcal{F}_{G_\Gamma,K}$ \cite[\S2]{Malle2010}.  The subfamily $\mathcal{F}_\Gamma^1 \subset \mathcal{F}_\Gamma$ is obtained by imposing finitely many local conditions (Proposition \ref{prop: hecke primes divide the resolvent disc}), so $\mathcal{F}_\Gamma^1$ either has positive relative density in $\mathcal{F}_\Gamma$ or is empty.  If $\mathcal{F}_\Gamma^1$ is empty, our heuristics again simply say that the distribution on $\mathcal{F}_\Gamma$ should match that on $\mathcal{F}_{G_\Gamma,K}$.  When $\mathcal{F}_\Gamma^1$ is non-empty, we say that $\Gamma$ is {\it aberrant}.  For aberrant $\Gamma$, our predicted distribution differs from Malle's distribution for $\mathcal{F}_{G_\Gamma,K}$. 

The simplest example of an aberrant $\Q$-group is $\mu_p$ (for a prime $p > 2$). The corresponding family $\mathcal{F}_{\mu_p}$ is the family of pure fields $\Q(\sqrt[p]{n})$, and $\mathcal{F}_{\mu_p}^1$ is the subfamily where $\Q(\sqrt[p]{n})$ is wildly ramified at $p$. Thus, we predict that the  dichotomy proven in Theorem \ref{main thm: over Q} persists for all primes $p > 2$.  In \S\ref{sec: Gamma-extension heuristics}, we give other examples of aberrant $K$-groups $\Gamma$. We suspect that there are only finitely many aberrant $K$-groups $\Gamma$ with a given Galois group $G_\Gamma$. For example, when $G_\Gamma = D_p$ the dihedral group of order $2p$, the aberrant $\Q$-groups correspond to the resolvent fields $\Q(i)$ and $ \Q(\sqrt{\pm p})$ (Proposition \ref{prop: dihedral hecke}).  We do not know a precise characterization of the aberrant $K$-groups in general. We note that it can happen that $\mathcal{F}_\Gamma = \mathcal{F}_{\Gamma}^1$, i.e.\ every field in $\mathcal{F}_\Gamma$ is Hecke unramified. For example, this happens for the family $F_n = \Q[x]/(x^3 + nx^2 + 9x + n)$ of cubic fields with quadratic resolvent $\Q(i)$. Thus, we predict that the average of $|\Cl_F[2]|$ in this family is $2$.  Another sample application of our heuristics (in the case $\Gamma = \mu_3$, see \S \ref{sec: kummer heuristics}) is:

\begin{conjecture} \label{conj: probability that h/h3 = 1}
As $n$ varies over primes congruent to $8 \pmod{9}$, the probability that 
$\Q(\sqrt[3]{n})$ has class number $1$ is
\begin{equation} \label{eqn: class number 1 probability tame}
    \prod_{q \neq 3} \prod_{j \ge 2} (1+q^{-j})^{-1},
\end{equation}
where $q$ ranges over the primes. On the other hand, as $n$ varies over primes congruent to $2,5 \pmod{9}$,  the probability that $\Q(\sqrt[3]{n})$ has class number $1$ is
\begin{equation} \label{eqn: class number 1 probability wild}
    \frac23\cdot 
\prod_{q\neq3} \prod_{j \ge 2} (1+q^{-j})^{-1}.
\end{equation}
\end{conjecture}
We restrict to $n \equiv 2\pmod{3}$, since the class number is divisible by $3$ when $n \equiv 1\pmod{3}$, by genus theory. Using PARI/GP \cite{PariGp} we find that (\ref{eqn: class number 1 probability tame}) is $\approx 0.5662$, while (\ref{eqn: class number 1 probability wild}) is $\approx 0.3775$. This makes precise the approximate probabilities observed in the earlier quotation from \cite{CohenMartinet87}.

\subsection{Connection to spin structures} \label{sec: connection to spin structures}
% The aberrant families $\mathcal{F}_\Gamma^1$ are characterized by a lack of Hecke primes. Equivalently, the Hecke ideal $H_{F/K}$ admits an algebraically varying square root in the family. The latter formulation fits into the more general framework of arithmetic spin structures developed by the second author and Venkatesh in \cite{SiadVenkatesh}. 

A {\it spin structure} $\sigma$ on a number field $F$ is a square-class $[\alpha] \in F^\times/F^{\times 2}$ such that $$(\alpha)\Diff_{F/\Q}^{-1} = I^2 $$ for some fractional ideal $I$. Spin structures exist, by Hecke's theorem. 
% The terminology comes via analogy with $3$-manifolds and Riemann surfaces, where spin structures are well-known to give rise to quadratic refinements of the natural bilinear structure on first homology, see \cite{KirbyTaylor} and \cite{AtiyahSpin,MumfordTheta,Johnson} respectively.  
We say a family of number fields $\mathcal{F}$ is {\it gyroscopic} if it admits a family-wide spin structure. 

In forthcoming work, the second author and Venkatesh \cite{SiadVenkatesh} give the following heuristic for the distribution of class groups in gyroscopic families
$\mathcal{F}$ of $K$-extensions with fixed and absolutely irreducible Galois group $G \subset S_m$: for any finite abelian $2$-group $H$, we have 
\begin{equation} \label{eqn: SV distribution}
   \mathbb{E}_{F \in \mathcal{F}}(|\mathrm{Surj}(\Cl_{F/K},H)|) =  \frac{\lvert H[2] \rvert \cdot  \lvert \wedge^2 H[|\mu_\infty(K)|]\rvert}{|H|^u},
\end{equation}
where $u$ is the Cohen--Martinet $u$-invariant. We refer to \cite{SiadVenkatesh}, where it is explained how spin structures induce quadratic refinements on various bilinear pairings defined on the class group and how the induced quadratic biextension structure leads to the above heuristic.

The wild family $F_n = \Q(\sqrt[3]{n})$ of Theorem \ref{main thm: over Q}(1) is gyroscopic. Indeed, we have $(-3)\Diff^{-1}_{F/\Q} =  I^2$, giving the family-wide spin structure $\sigma = -3.$  More generally, if $\Gamma$ is an aberrant $\Q$-group, and if $\mathcal{F}_\Gamma$ admits a universal polynomial $f$ (over an appropriate base space), then the family $\mathcal{F}_\Gamma^1$ is gyroscopic with spin structure $\sigma = \Disc(f)$. The prediction $(\ref{eqn: SV distribution})$ is then easily seen to imply the predictions in \S \ref{subsec: Gamma extension heuristics intro} for the families $\mathcal{F}_\Gamma^1$. We see that the predictions of \cite{SiadVenkatesh} for $\Cl_{F/\Q}[2]$ in families of the form $\mathcal{F}_\Gamma$ can be derived directly from Hecke reciprocity. Of course, (\ref{eqn: SV distribution}) is a more general prediction for the larger group $\Cl_{F/\Q}[2^\infty]$; e.g.\ it predicts the probability that the prime-to-$3$ part of $|\Cl_{\Q(\sqrt[3]{n})}|$ should be any given integer, generalizing Conjecture \ref{conj: probability that h/h3 = 1}. We explain how to extract this from (\ref{eqn: SV distribution}) in \S \ref{subsec: spin conjecture derivation} and offer numerical evidence in Appendix \ref{sec: appendix tables}.

The families of $n$-monogenized fields studied by Bhargava, Hanke, Shankar, Siad, and Swaminathan \cite{BhargavaHankeShankar, 2011.08834, 2011.08842, Swaminathan, ShankarSiadSwaminathan} are also (partly) gyroscopic, though in a different way. The spin structure in this case is $\sigma = nf'(\beta)$, where $\beta$ is an $n$-monogenizer and $f$ is its minimal polynomial. These works show that these families also exhibit boosts in the distribution of their class groups. Thus, spin structures give a common framework for understanding how algebraic properties of families of number fields impact the distribution of the class group.

\subsection{Previous work}
Ruth, in his Ph.D. thesis, proves the upper bound in Theorem \ref{main thm: over Q}(1) for the special sub-family where $n \not\equiv \pm 1\pmod{9}$ and $n$ is not just cubefree, but squarefree \cite[Theorem 1.1.5]{RuthThesis}. For such $n$, the element $\sqrt[3]{n}$ is a monogenizer, allowing Ruth to sidestep the subtle issue of parameterizing elements of $\Sel_2^\un(F_n/\Q)$ for all $n$, which is intimately tied to our notion of Hecke reciprocity.  In the language of spin structures, Ruth considers the subfamily of pure cubic fields which are gyroscopic in two different ways. Interestingly, the additional spin structure does not boost the average size of the $2$-torsion further; evidently what matters is the {\it existence} of a family-wide spin structure, not the number of such spin structures.  Note that Ruth's subfamily has density $0$ when the $F_n$ are ordered by discriminant, but has positive density when the $F_n$ are ordered by $n$. 

For more speculation on the proportions proven in Theorem \ref{main thm: over Q}, see \cite[\S4-6]{CohenMartinet94}. Note that \S\ref{subsec: Gamma extension heuristics intro} provides an answer to Problem 6.1 and Question 6.2 of {\it loc.\ cit}.  In \cite{CohenMartinet87}, they remark that when the fields $F_n = \Q(\sqrt[3]{n})$ are ordered by discriminant instead of by $n$, the data for $\Cl_{F_n}[2]$ (for all $n$) matches better with the Cohen--Martinet distribution for general $S_3$-cubic fields. It seems more likely to us that this is because the ratio of tame fields to wild fields is $8/3$ in the large discriminant limit,  and  moreover there is a large second order term that makes this disparity more pronounced for small discriminants. This follows from \cite[Theorem 8]{BhargavaShnidman14}, which divides the family of pure cubic fields into two subfamilies based on the ``shape'' of $\O_{F_n}$, and these two subfamilies happen to exactly be the subfamilies $\mathcal{F}_{\mu_3}^1$ and $\mathcal{F}_{\mu_3}^2$.  The division of pure cubic fields into ``type I'' and ``type II'' in fact goes back to Dedekind, who observed in \cite{Dedekind1900} that $\O_{F_n}$ takes on a different algebraic form, depending on whether $F_n$ is wildly or tamely ramified at $3$.

\subsection{Related work}
While posting this paper, we were made aware of the recent work of Loughran--Santens \cite{LoughranSantens} on Malle's conjecture. It seems plausible that one can extract a prediction from their work for the averages computed in Theorems \ref{thm: over Q(omega)} and \ref{main thm: over Q}, and more generally for the distribution of $\Cl_{F/K}[2^\infty]$ in the families covered by our heuristics. It will be interesting to see if these predictions agree and how Hecke reciprocity relates to their Brauer group computations.

\subsection{Acknowledgements}
We thank Will Sawin for discussions related to the proof of Theorem \ref{thm: reciprocity law for odd degree intro}. We also thank Manjul Bhargava, Bhargav Bhatt, Nick Katz, Dan Loughran, Tim Santens, Peter Sarnak, Sameera Vemulapalli and Akshay Venkatesh for helpful conversations and comments. Shnidman was partially funded by the European Research Council (CurveArithmetic, 101078157), as well as the Ambrose Monell Foundation while at the Institute of Advanced Study.  Siad was supported by the National Science Foundation under Grant No.\ DMS-1926686 at the Institute for Advanced Study. He also thanks Princeton University for providing an excellent working environment. 

\subsection{Notation}
\begin{itemize}
    \item $K$ will generally denote a field of characteristic $0$.  
    \item Denote by $\overline{K}$ an algebraic closure and write $\Gal_K = \Gal(\overline{K}/K)$.  
    \item For a finite extension $F/K$, write $\tilde F/K$ for its normal closure.
    \item If $K$ is a number field,  $\Cl_K$ denotes the ideal class group of the ring of integers $\O_K$.
    \item Write $\Cl_{F/K}$ for the cokernel of the natural map $\Cl_K \to \Cl_F$. 
    \item $\mu_n$ denotes the group of $n$-th roots of unity in $\overline{K}^\times$. 
    \item For a finite group $G$, let $\underline{G}$ be the $K$-group $G$ with trivial $\Gal_K$-action. 
    \item For a finite $K$-group $\Gamma$, let $\Gamma_0$ be its underlying abstract group. 
    \item For an abelian group $G$, let $G^\vee = \Hom(G, \Q/\Z)$.

\end{itemize}

\section{Unramified Selmer groups}
Fix a prime $\ell$ and an extension of number fields $F/K$. When $\mu_\ell \subset K$ and $\ell\nmid [F \colon K]$, we show how to view $\Cl_{F/K}^\vee[\ell]$ as a Selmer group. See \cite{DummitVoight} and \cite{LipnowskiSawinTsimerman} for two related perspectives.

\subsection{Absolute Selmer groups}
Set $T = H^1(F, \mu_\ell)$ and identify it with $F^\times/F^{\times \ell}$ via Kummer theory. Similarly, for any place $w$ of $F$, set $T_w = H^1(F_w, \mu_\ell)$ and define the following subspace
\[U_w = \begin{cases}    
\O_{F_w}^\times/\O_{F_w}^{\times \ell} & \mbox{ if $w$ non-archimedean }\\ 
 T_w  & \mbox{ if $w$ archimedean}
\end{cases}.\] 
The traditional $\ell$-Selmer group of $F$ is then
\[\Sel_\ell(F) = \ker\left(T \to \prod_w T_w/U_w\right),\]
where the arrow is induced by the restriction maps $T \to T_w$. 
 Concretely, $\Sel_\ell(F)$ is the subgroup of $[t] \in F^\times/F^{\times \ell}$ with valuation divisible by $\ell$ at all finite places, i.e.\ $(t) = I^\ell$ for some ideal $I$.  The map $[t] \mapsto [I]$ gives rise to the short exact sequence
\begin{equation}\label{seq: selmer group ses}
    0 \to \O^\times_F/\O_F^{\times \ell} \to \Sel_\ell(F) \to \Cl_F[\ell] \to 0.
    \end{equation}

Now consider the following slightly different Selmer structure on $\mu_\ell$. For a finite place $w$, let $\epsilon_w \in T_w$ be a non-trivial unramified class (unique up to scaling), and define 
\[S_w = \begin{cases}
    \langle \epsilon_w\rangle & \mbox { if $w$ non-archimedean} \\
    \{0\} & \mbox{ if $w$ archimedean}.
\end{cases}\]
Note that if $F$ contains $\mu_\ell$, then  $S_w = U_w^\perp$ with respect to the Hilbert symbol $(\, ,\, )_w$. 

Define the unramified Selmer group
\[\Sel_\ell^\un(F) = \ker\left(T \to \prod_wT_w/S_w \right).\]

\begin{theorem}\label{thm: canonical isomorphism}
If $F$ contains $\mu_\ell$, there is a canonical isomorphism $\Sel_\ell^\un(F) \stackrel{\xi}{\longrightarrow} \Cl_F^\vee[\ell]$. 
\end{theorem}
\begin{proof}
% The existence of an isomorphism can be seen from Kummer theory and class field theory, but let us write down the canonical isomorphism. 
Class field theory gives an identification 
\[\Cl_F^\vee[\ell] \simeq \Hom_{\mathrm{cts}}(F^\times\backslash\A_F^\times/\prod_w \O_{F_w}^\times, \Z/\ell\Z),\] where the product is over the finite places of $F$. 
Define $\xi \colon \Sel_\ell^\un(F) \to \Cl_F^\vee[\ell]$ by the formula $\xi(t) = \sum_w ( - ,  t )_w$. This is well-defined because of the local conditions defining $\Sel_\ell^\un(F)$ and the product formula for the Hilbert symbol.  Note that $\xi$ is injective by strong approximation and the local-to-global principle for $\ell$-th powers. On the other hand, the Greenberg-Wiles formula \cite[Theorem 8.7.9]{NeukirchSchmidtWingberg} and $(\ref{seq: selmer group ses})$ give $\dim \Sel_\ell^\un(F) = \dim \Cl_F[\ell]=\dim \Cl_F^\vee[\ell]$. Thus, $\xi$ is an isomorphism. 
\end{proof}

Note that $\Sel_\ell^\un(F)$ is a subgroup of $\Sel_\ell(F)$. The composition 
\begin{equation}\label{eq: LST map}
    \Cl_F^\vee[\ell] \stackrel{\xi^{-1}}{\longrightarrow} \Sel_\ell^\un(F) \hookrightarrow \Sel_\ell(F) \to \Cl_F[\ell]
\end{equation}
is the map called $\psi_F$ in \cite{LipnowskiSawinTsimerman}.

\subsection{Relative Selmer groups}\label{subsec: relative signature groups}
Let $K$ be a number field containing $\mu_\ell$ and let $F/K$ be a finite extension.
\begin{lemma}
    The norm map $F^\times/F^{\times \ell} \to K^\times/K^{\times \ell}$ sends $\Sel_\ell^\un(F)$ to $\Sel_\ell^\un(K)$.  
\end{lemma} 
\begin{proof}
    We need to check that for a place $w$ of $F$ above a finite place $v$ of $K$, that the norm of $\epsilon_w$ lies in the span of $\epsilon_v$. For all $u \in \O_{K,v}^\times$, we have $\langle \Nm(\epsilon_w), u \rangle_v = \langle \epsilon_w, u\rangle_w = 1$ \cite[pg. 209]{Serre1979}, so $\Nm(\epsilon_w)$ is orthogonal to all $v$-units and hence lies in $\langle \epsilon_v \rangle$.
\end{proof}
Define 
\[\Sel_\ell^\un(F/K) = \ker\left(\Sel_\ell^\un(F) \stackrel{\Nm}{\longrightarrow} \Sel_\ell^\un(K)\right).\]
If $\Cl_K[\ell] = 0$, then $\Sel_\ell^\un(F/K) \simeq \Sel_\ell^\un(F)$. However, even in this case, the relative Selmer framework gives the better point of view. For example, a class in $\Sel_2^\un(F)$ is locally in the span of $\epsilon_w$, but it must also lie in the kernel of the local norm, which may not be true about $\epsilon_w$ itself.

Returning to the general case $F/K$, we are led to define the $\Gal_K$-module 
\[M := \ker\left(\Res_{F/K}(\mu_\ell) \stackrel{\Nm}{\longrightarrow} \mu_\ell\right)\]
of order $\ell^{[F \colon K] - 1}$, whose group of geometric points is the subgroup of elements in $\mu_\ell^{[F \colon K]}$ such that the product of all entries is $1$.  We set $W = H^1(K, M)$ and make use of the identification  
\[W \simeq \ker\left(F^\times/F^{\times\ell} \stackrel{\Nm}{\longrightarrow}K^\times/K^{\times \ell}\right),\]
again coming from Kummer theory.
For any place $v$ of $K$, write $W_v = H^1(K_p, M)$. 

We define a Selmer structure on $W$ using the subgroups $V_v =W_v \cap \prod_{w \mid v} S_w$. The following follows from the definitions:

\begin{lemma}
    The Selmer group $\ker(W \to \prod_v W_v/V_v)$ is isomorphic to $\Sel_\ell^\un(F/K)$.
\end{lemma}

Recall that $\Cl_{F/K} = \Cl_F/\Cl_K$. 
\begin{lemma}
    If $\ell \nmid [F \colon K]$, then $\Sel_\ell^\un(F/K) \simeq \Cl_{F/K}^\vee[\ell]$.
\end{lemma}
\begin{proof}
    Since $\ell \nmid [F \colon K]$, the group $\Cl_{F/K}[\ell] = \mathrm{coker}(\Cl_K[\ell] \stackrel{i}{\longrightarrow} \Cl_F[\ell])$ is dual to the Selmer group $\Sel_\ell^\un(F/K) = \ker(\Cl_F^\vee[\ell] \stackrel{\Nm}{\longrightarrow}\Cl_K^\vee[\ell])$ via Theorem \ref{thm: canonical isomorphism}. 
\end{proof}

\begin{example}\label{ex: local triviality}
{\em 
    Suppose $\ell \nmid [F \colon K]$, and let $v$ be a prime of $K$ such that $F/K$ has a unique prime $w$ above $p$. Then the subgroup $V_v$ is trivial (since $\epsilon_v = \epsilon_w$ is not in the kernel of the local norm) and indeed every element of $\Sel_\ell^\un(F/K)$ is locally trivial at $w$ in this situation.  
    }
\end{example}

More generally, we have:
\begin{lemma}\label{lem: local selmer ratio}
    If $\ell \nmid [F: K]$ and $v$ is a finite prime of $K$, then  $|V_v| = |M(K_v)| = \ell^{N-1}$, where $N$ is the number of primes of $F$ above $v$. 
\end{lemma}
\begin{proof}
    Write $F \otimes K_v \simeq \prod_{i = 1}^N L_i$, where each $L_i$ is a field extension of $K_v$ corresponding to a prime $w_i$ above $v$. On one hand, we have $|M(K_v)| = \ell^{N-1}$. On the other hand, there exists at least one $L_i$ such that $\ell \nmid [L_i \colon K]$, and we may label things so that $i = 1$. Then note that $\epsilon_1$ is the image of the unramified class $\epsilon \in K_v^\times/K^{\times \ell}_v$. Thus $\Nm(\epsilon_1, 1, \ldots, 1) = \Nm(\epsilon_1) \cdot 1 = \epsilon_1^{[L_i \colon K]}  = \epsilon$, i.e.\ the norm map $\prod_i L_i^\times/L_i^{\times \ell} \to K_v^\times/K_v^{\times \ell}$ induces an isomorphism on unramified classes when restricted to the first factor. Projection onto the last $N-1$ factors therefore induces an isomorphism $V_v \simeq \prod_{i = 2}^N S_{w_i}$, and hence $|V_v| = \ell^{N-1}$.  
\end{proof}

From now on, we view $\Sel_\ell^\un(F/K)$ as the Selmer group of the Selmer structure $(M, \{V_v\})$. Similarly for $\Sel_\ell^\un(F)$ whenever $\Sel_\ell^\un(K)$ (or equivalently $\Cl_K[\ell]$) happens to be $0$.  

\section{Hecke primes and Hecke reciprocity}\label{sec: reciprocity laws}

Let $K$ be a local or global field of characteristic $0$. We study the notion of Hecke primes and Hecke ramification in extensions $F/K$ of odd degree.  In the global case, we prove a reciprocity law constraining the number of Hecke primes of $F$ that are inert in an unramified quadratic extension $L/F$ corresponding to a class in $\Cl_{F/K}^\vee[2]$.

\subsection{Hecke primes over local fields}\label{subsec: Hecke primes local fields}
Let $K$ be a finite extension of $\Q_p$ for some prime $p$, with ring of integers $\O$. Let $v$ be the discrete valuation on $K$, normalized so that $v(\pi) = 1$, where $\pi\in \O$ is any uniformizer.  Let $F$ be an \'etale $K$-algebra of degree $m$, with $m = 2g + 1$ odd.  

We may write $F \simeq \prod_{i = 1}^N L_i$, where each $L_i$ is a finite field extensions of $K$. As before, set
\[M := \ker\left(\Res_{F/K}(\mu_\ell) \stackrel{\Nm}{\longrightarrow} \mu_\ell\right),\]
let 
\[W = H^1(K, M) \simeq \ker(F^\times/F^{\times 2} \to K^\times/K^{\times 2}),\] 
and let $V \subset W$ be the subgroup of classes $[t] = (u_1,\ldots, u_N)$ such that each $u_i$ is either trivial or the unramified square-class $\epsilon_i$ in $L_i$.

We define the {\it Hecke ideal} to be
\[H_{F/K} := \Disc_{F/K}\mathrm{Diff}_{F/K}^{-1} \subset \O_F.\]
Its norm $\Disc_{F/K}^m\Disc_{F/K}^{-1} = \Disc_{F/K}^{2g}$ is a square of an ideal. We say that $F/K$ is {\it Hecke ramified} if $H_{F/K}$ is {\it not} a square of an ideal. Otherwise, $F/K$ is {\it Hecke unramified}.  If $F/K$ is Hecke ramified then $F/K$ is ramified in the usual sense, but the converse need not be true. Finally, we say that a prime $\m_i$ (corresponding to $L_i$) of $\O_F$ is a {\it Hecke prime} if the power of $\m_i$ dividing $H_{F/K}$ is odd and if there exists $[t] = (u_1,\ldots,u_N) \in V$ such that $u_i \notin L_i^{\times2}$. 

We record below some basic facts about Hecke primes; the proofs will use the following notation. Write $(\pi) = \prod_i \m_i^{e_i}$ and let $f_i$ be the residue field degree at $\m_i$.  Write $\Diff_{F/K} = \prod_i \m_i^{d_i}$, so that $\Disc_{F/K} = (\pi)^d$ where $d = \sum_i d_if_i$. In this notation, $H_{F/K} = \prod_i \m_i^{de_i - d_i}$ and of course $m = \sum_i e_i f_i$. We write $h_i = de_i - d_i$ for the exponent of $\m_i$ in the Hecke ideal.

\begin{lemma}\label{lem: unram or tot ram}
    If $F/K$ is either unramified or a field extension, then $F/K$ is Hecke unramified.
\end{lemma}

\begin{proof}
    The unramified case is clear, so assume $F$ is a field. Then $(\pi) = \m_1^{e_1}$ and $H_{F/K} = \m_1^{d_1f_1e_1 - d_1} = \m_1^{d_1(e_1f_1 - 1)} = \m_1^{d_1(m-1)}$, which is a square.
\end{proof}

\begin{lemma}\label{lem: hecke ramified implies hecke prime}
Suppose $F/K$ is Hecke ramified. Then there exists a Hecke prime $\m_i$ of $F/K$.     
\end{lemma}

\begin{proof}
    The hypothesis implies that there exists a prime $\m_1$ such that $h_1$ is odd. If there does not exist $[t] = (t_1,\ldots, t_N) \in V$ with $t_1 \notin L_1^{\times 2}$, it means that $e_i$ is even for $2\leq i \leq N$ (see the proof of Lemma \ref{lem: local selmer ratio} and use the fact that $\Nm(\epsilon_{L_i}) = \epsilon_K$ if and only if $e_i$ is odd). Since $m = \sum_i e_if_i$ is odd, it follows that $e_1f_1$ is odd and hence both $e_1$ and $f_1$ are odd. We may also assume that $h_i$ is even for $2 \leq i \leq N$ since at most one prime with $h_i$ odd can fail to be a Hecke prime (again by the proof of Lemma \ref{lem: local selmer ratio}).
    
    Suppose first that $d = \sum_i d_i f_i$ is even. Then $h_i \equiv d_i \pmod{2}$ for all $i$. It follows that $d_1 \equiv h_1$ is odd and $d_i \equiv h_i$ is even for all $i \geq 2$. But then $d  = \sum d_if_i \equiv d_1f_1$ is odd, which is a contradiction. Now suppose $d = \sum_i d_if_i$ is odd.  Then $h_i \equiv e_i - d_i$ for all $i$, and so $e_1 \not\equiv d_1$, and hence $d_1$ is even. But then $e_i \equiv d_i$ for $i \geq 2$, which implies $d_i$ is also even for $i \geq 2$. This implies that $d = \sum d_if_i$ is even, which is a contradiction.  
\end{proof}

As an example, here is a characterization of Hecke ramification in the cubic case: 

\begin{lemma}\label{lem: cubic characterization}
If $m = 3$, then $F/K$ is Hecke ramified if and only if $F \simeq K(\sqrt{\pi}) \times K$, for some uniformizer $\pi$. In this case, the ramified prime $\m_1$ is the unique Hecke prime of $F/K$.    
\end{lemma}

\begin{proof}
    If $F/K$ is unramified or totally ramified, then $F/K$ is Hecke unramified by Lemma \ref{lem: unram or tot ram}. If $F/K$ is partially ramified and $p \neq 2$, then $(\pi) = \m_1^2\m_2$, and $H_{F/K} = \m_1^{2-1}\m_2 = \m_1\m_2$. Thus $F/K$ is Hecke ramified in this case, and moreover $L_1 \simeq K(\sqrt{\pi})$ for some uniformizer $\pi$, since $p \neq 2$.
    
    If $F/K$ is partially ramified and $p = 2$, then $\Disc(F/K) = (\pi)^e$ for some $e \geq 2$, and $H_{F/K} = \m_1^{2e - e}\m_2^e = (\m_1\m_2)^e$ is a square if and only if $e$ is even. Write $F \simeq K(\sqrt{d}) \times K$, for some $d \in K$. We claim that $e$ is even if and only if $d \in \O_K^\times$ (up to squares in $K^\times$). By \cite[p.\ 64]{Serre1979}, we have $e = w(\overline{\alpha} - \alpha)$, where $w$ is the normalized valuation on $K(\sqrt{d})$ and $\alpha$ is such that $\O_{K(\sqrt{d})} = \O_K[\alpha]$. Writing $\alpha = a + b\sqrt{d}$, we have $w(\overline{\alpha} - \alpha) = w(2b\sqrt{d}) \in 2\Z + w(\sqrt{d})$, which is even if and only if $d$ is a unit (up to squares), as claimed. 

    Finally, in the Hecke ramified case, the group $V$ has order $2$ (by Lemma \ref{lem: local selmer ratio}) and the non-trivial class is $t = (\epsilon_1, 1)$. It follows that $\m_1$ corresponding to $K(\sqrt{d}) = L_1$ is the unique Hecke prime. 
\end{proof}

\begin{example}
    {\em 
    In higher degree, it can happen that $F/K$ is Hecke ramified but there are no {\it ramified} Hecke primes. For example, if $F/\Q_3$ has degree five and splitting type $(1^32)$,  then $d = d_1f_1 = d_1 \in \{3,5\}$ is odd. Hence $(\pi) = \m_1^3\m_2$ and 
    \[H_{F/\Q_3} = \m_1^{h_1}\m_2^{h_2} = \m_1^{de_1 - d_i}\m_2^{de_2 - d_2} = \m_1^{2d}\m_2^{d},\]
    which shows that the unramified prime $\m_2$ is the unique Hecke prime for $F/\Q_3$.
    }
\end{example}
\subsection{Hecke reciprocity}\label{subsec: Hecke reciprocity}
Now suppose that $K$ is a number field and let $F/K$ be a finite extension. Recall the following theorem of Hecke \cite[pg. 291]{WeilBasicNumberTheory}.

\begin{theorem}\label{thm: hecke}
    The different $\mathrm{Diff}_{F/K}$ is a square in $\Cl_F$.
\end{theorem}

As in the introduction, we define the {\it Hecke ideal} 
\[H_{F/K} := \Disc_{F/K} \mathrm{Diff}_{F/K}^{-1} \subset \O_F.\] The Hecke ideal is especially interesting when $F/K$ has odd degree, because then the norm of $H_{F/K}$ is the square of an ideal in $\O_K$ (while $H_{F/K}$ itself need not be the square of an ideal in $\O_F$).  As the discriminant $\Disc_{F/K}$ is a square in $\Cl_K$, Hecke's theorem also gives:

\begin{corollary}\label{cor: hecke ideal}
    The Hecke ideal $H_{F/K}$ is a square in $\Cl_F$.
\end{corollary}

To state the reciprocity law, we recall the definition of a Hecke prime from the introduction (which is also the global version of the local definition given in the previous subsection). 
\begin{definition}
{\em 
A prime $w$ of $F$ lying over a finite prime $v$ of $K$, is a {\it Hecke prime} if an odd power of $w$ divides $H_{F/K}$ and if there exists $t \in V_v$ which projects to $\epsilon_w$ in $S_w$ (i.e.\ $t \notin F_w^{\times 2}$). 
}
\end{definition}

The existence of a Hecke prime is an obstruction to $H_{F/K}$ being a square of an ideal.\footnote{Analogously, a polarization $\lambda \colon A \to \hat{A}$ of an abelian variety over a field $k$ always has square degree, but it may not arise from a line bundle defined over $k$. On the other hand, $2\lambda$ does have this property \cite[\S20]{MumfordAbelianVarieties}.}  Hecke reciprocity (which is also Theorem \ref{thm: reciprocity law for odd degree intro}) states:

\begin{theorem}\label{thm: hecke reciprocity}
    Let $F/K$ be a finite extension of number fields and let $L = F(\sqrt{t})/F$ be an everywhere unramified quadratic extension such that $\Nm_{F/K}(t) \in K^{\times 2}$.  Then the number of Hecke primes of $F$ that are inert in $L/F$ is even.
\end{theorem}

\begin{proof}
    Let $\chi \colon \Cl_F \to \Z/2\Z$ be the quadratic character corresponding to $L/F$. For every prime $w$ of $F$, we have $\chi(w) = 0$ if and only if $w$ splits in $L/K$ if and only if $t \notin F_w^{\times 2}$. By Theorem \ref{cor: hecke ideal}, $\chi([H_{F/K}]) = 0$. Write $H_{F/K} = \prod_w w^{d_w}$ and let $S$ be the set of Hecke primes $w$ that are inert in $L/F$.  Then 
    \[0 = \chi([H_{F/K}]) = \sum_{\mbox{$w$ inert in $L/F$}} d_w= \sum_{w \in S} 1 \in \Z/2\Z,\] 
which means that $\#S$ is even. 
\end{proof}

Here is a sample application of the reciprocity law.

\begin{theorem}\label{thm: Kummer obstruction}
    Let $p$ be an odd prime and let $F/\Q$ be a degree $p$ Kummer extension $F = \Q(\sqrt[p]{n})$. Then any prime of $F$ lying above $p$ is split in any unramified quadratic extension $L/F$. 
\end{theorem}
\begin{proof}
Since $\Cl_\Q[2] = 0$, every unramified quadratic extension $L/F$ arises from some $t \in \Sel_2^\un(F/\Q)$.  There are two possibilities for the splitting behavior of $p$ in $F$: either $p = \lambda^p$ or $(p) = \lambda_1^{p-1} \lambda_2$. In the first case we have $V_p = 0$ by Lemma \ref{lem: local selmer ratio}, so $\lambda$ splits in all quadratic extensions $L/F$ for local reasons. In the second case, we have $\#V_p = 2$ and the non-trivial class is $(\epsilon_{\lambda_1},1)$. This shows that $\lambda_2$ is not Hecke and moreover splits in every unramified $L/F$ for local reasons. On the other hand, $\lambda_1$ is a Hecke prime, and we claim that it is the unique Hecke prime in $F/\Q$. This is because any ramified prime $q \neq p$ is totally ramified, and hence $V_q = 0$ by Lemma \ref{lem: local selmer ratio}.  By Theorem \ref{thm: hecke reciprocity}, we see that  $\lambda_1$ must also split in $L/F$ (for global reasons!).   
\end{proof}

More generally, the same proof gives:

\begin{corollary}
    Suppose $F/K$ is an odd degree extension with exactly one Hecke prime $w$. Then each $[t] \in \Sel_2^\un(F/K)$ is locally a square at $w$.  
\end{corollary}

Using the formulas in \cite{BhargavaShnidman14}, we can write down many other examples of families of fields admitting at most one Hecke prime:

\begin{example}\label{ex: one Hecke ramified prime}
{\em 
    Let $p$ be a non-zero integer with at most one prime factor. For $n \in \Q^\times$, set $f_n = x^3 + nx^2 - 9px - np$ and  $F_n = \Q[x]/f_n(x)$. The discriminant of $f_n$ is $4p(n^2 + 27p)^2$, so this is a family of cubic fields with quadratic resolvent $\Q(\sqrt{p})$. If $p = \pm 1$, there are no Hecke primes in this family, while if $|p| > 1$, then $p$ is the only prime which is Hecke ramified in $F_n$. 
    }
\end{example}

\section{Quadratic refinements}

To prove our results on cubic Kummer extensions $F = K(\sqrt[3]{n})$, we analyze a certain quadratic refinement of the pairing on $\Sel_2^\un(F/K) \times \Sel_2^\un(F/K) \to \Br(K)[2]$ induced by cup product. In the process, we will essentially reprove the Hecke reciprocity law from the previous section, without invoking Hecke's theorem. Since it may be of independent interest, we use quadratic refinements to give an algebraic proof of Hecke's theorem for general cubic extensions.

\subsection{Quadratic refinements}\label{subsec: the obstruction map}
 We begin in the general setting where $K$ is a field of characteristic $0$ and $F$ is a finite \'etale $K$-algebra of odd degree $m = 2g+1$.

 Recall the $K$-group scheme $M = \ker(\Res_{F/K}\mu_2 \stackrel{\Nm}{\longrightarrow} \mu_2)$. We define a bilinear pairing 
 \[\langle \,  , \, \rangle \colon M \times M \to \mu_2\]
 as follows. Let $M^0 = \Res_{F/K}\mu_2$. For $i = 1, \ldots, m$, let $\sigma_i \colon K \to \overline{K}$ be the different embeddings of $F$ into $\overline{K}$. With this ordering, we identify $M^0(\overline{K}) \simeq \mu_2^m$. Let $\Delta = (-1,\ldots, -1) \in \mu_2^m$ and let $e_1, \ldots, e_m$ be the standard basis of $\mu_2^m$. The pairing on $M^0(\overline{K})$ determined by the rule 
 \[\langle e_i,e_j \rangle = \begin{cases} 1 & \mbox{ if $i = j$}\\
-1 & \mbox{ if $i \neq j$}
 \end{cases}
 \] is $\Gal_K$-equivariant, hence descends to a pairing $M^0 \times M^0 \to \mu_2$. 
 
\begin{lemma}\label{lem: self-dual}
    $M$ is isomorphic to its own Cartier dual via $\langle \, , \,  \rangle$.
\end{lemma}
\begin{proof}
    The kernel for the pairing on $M^0$ is the span of $\Delta = (-1,\ldots, -1)$, and so the induced pairing on $M$ is perfect. 
\end{proof}

Cup-product and the self-duality of Lemma \ref{lem: self-dual} induce a bilinear form 
\[( \, , \, ) \colon H^1(K, M) \times H^1(K,M) \to H^2(K, \mu_2)[2] = \Br(K)[2]\]
valued in the $2$-torsion subgroup of the Brauer group of $K$. When $F/K$ is an extension of number fields, this pairing is given (locally at each prime) by the Hilbert symbol.

A {\it quadratic refinement} of $( \, , )$ is a quadratic map $q \colon H^1(K, M) \to \Br(K)[2]$ such that 
\[( x, y ) = q(x+y) - q(x) - q(y).\]
There are many quadratic refinements of $(\, ,\, )$, each one depending on a choice of $\beta \in F^\times$ such that $\Tr_{F/K}(\beta) = 0$. Let $f(x) \in K[x]$ be the monic, degree $m$ characteristic polynomial of $\beta$ acting by multiplication on $F$. To $t \in H^1(K, M) \subset F^\times/F^{\times2}$, we attach the quadratic space $(F, Q_{\beta,t})$, where 
\[Q_{\beta,t}(x,y)=\Tr_{F/K}(txy/f'(\beta)).\] Because $f'(\beta)$ generates the different of $\O_K[\beta]$ and because $t$ has square norm, this quadratic space of rank $2g+1$ has discriminant $1$. We say $Q_{\beta,t}$ is {\it split} if it contains a $g$-dimensional isotropic subspace. This is the case if $t = 1$, and in this case the orthogonal group is the split group $\SO(m)$.  

\begin{proposition}\label{prop: obstruction in terms of isotropy}
    For each $\beta$ with $\Tr_{F/K}(\beta) = 0$, there exists a quadratic refinement 
    \[q_\beta \colon H^1(K, M) \to \Br(K)[2]\] 
    sending $[t]$ to the class of $Q_{\beta, t}$ in $H^1(K, \SO(m)) \simeq \Br(K)[2]$. In particular,  $[t] \in H^1(K,M)$ lies in $\ker(q_\beta)$ if and only if $Q_{\beta,t}$ is split.  
\end{proposition}

\begin{proof}
    This is nicely explained in Bhargava-Gross \cite[\S4.2]{BhargavaGross2013}, see also \cite[\S6.2]{BhargavaGross2014} and \cite[Corollary 4.6]{PoonenRains}. 
\end{proof}

\begin{remark}\label{rem: connecting map}
    {\em 
    One can view $q_\beta \colon H^1(K,M) \to \Br(K)[2]$ as the map $H^1(K, J[2]) \to H^2(K, \G_m)[2]$ in the long exact sequence associated to the theta group of the square of the principal polarization on the Jacobian $J$ of the hyperelliptic curve $y^2 = f_\beta(x)$, as explained in \cite[\S4]{PoonenRains}. 
    }
\end{remark}

In the next two sections, we analyze the maps $q_\beta$  over local and global fields, especially their interaction with the group $\Sel_2^\un(F/K)$ and its local analogue $V_p$. We give full details for cubic extensions and mostly full details for Kummer extensions $F = K(\sqrt[m]{n})$ of prime degree.

\subsection{Quadratic refinements over local fields}\label{subsec: quadratic refinements over local fields}
We return to the notation of Section \ref{subsec: Hecke primes local fields}, where $K$ is a finite extension of $\Q_p$ for some prime $p$ and $F/K$ is a degree $m = 2g+1$ \'etale $K$-algebra.

Let $\beta \in F$ be such that $\Tr_{F/K}(\beta) = 0$ and $F \simeq K[x]/f_\beta(x)$, where $f_\beta(x)$ is the characteristic polynomial of $\beta$, and let $q_\beta\colon W \to \Br(K)[2]$ be the quadratic refinement associated to $\beta$ as in Proposition \ref{prop: obstruction in terms of isotropy}. We show below that, in many cases, there exists $\beta$ such that $V \subset \ker(q_\beta)$ if and only if $F/K$ is Hecke unramified.  We begin with two statements that hold for any odd $m$. 

\begin{lemma}\label{lem: locally unobstructed at unramified primes}
    If $F/K$ is unramified, then $V\subset \ker(q_\beta)$, for any choice of $\beta$. 
\end{lemma}
\begin{proof}
    The assumptions imply that $M$ is unramified. Using Remark \ref{rem: connecting map},  we conclude that $q_\beta$ sends unramified classes to unramified classes. Since the non-zero $2$-torsion element of $H^2(K, \G_m)$ is ramified, we must have $q_\beta(t) = 0$ for all $t \in V$.
\end{proof}

\begin{lemma}\label{lem: field extensions}
    If $F/K$ is a field extension then $V\subset \ker(q_\beta)$, for any choice of $\beta$.
\end{lemma}
\begin{proof}
    In this case $V = 0$ by Lemma \ref{lem: local selmer ratio}. \end{proof}

To handle the more subtle cases, we assume for the rest of this section that $m = 3$. Then $q_\beta(t) = 0$ if and only if the rank $3$ quadratic form $\Tr(tx^2/f'(\beta))$ is isotropic. The following explicit criteria for isotropy follows from the definition of the Hilbert symbol $( \, , )$ on $K^\times/K^{\times 2} \times K^\times/K^{\times 2}$. 

\begin{lemma}\label{lem: isotropy criterion}
    Let $a,b,c \in K^\times$. Then $ax^2 + by^2 + cz^2$ is isotropic over $K$ if and only if $(-ac,-bc)= 1$. 
\end{lemma}

Since we now assume $m = 3$, the only case not covered by the earlier two lemmas  is when $F \simeq L \times K$, where $L$ is a ramified quadratic $K$-extension.

\begin{proposition}\label{prop: partially ramified cubic computation}
    Suppose $F = K(\sqrt{d}) \times K$, with $K(\sqrt{d})/K$ ramified. Then there exists a constant $v_0 \in \Z$ such that  if $\beta = (a + b\sqrt{d},-2a)$ satisfies $v(a) > v(b) + v_0$, then $V \subset \ker(q_\beta)$ if and only if $F/K$ is Hecke unramified. 
\end{proposition}

\begin{proof}
    A computation gives $f_\beta'(\beta) = (6ab\sqrt{d} + 2d b^2, 9a^2 - d b^2)$. If $v(a) \gg v(b)$, then we have $f'_\beta(\beta) \equiv (2d, -d) \equiv (2,-d)$ modulo squares in $F^\times$. There is a single non-trivial class in $V$, the class $(\epsilon, \Nm(\epsilon)) \equiv (\epsilon,1)$. Thus, we need to check whether the quadratic form $\Tr_{F/K}((2\epsilon, -d)x^2)$ is isotropic or not. An equivalent diagonalization of this quadratic form is $\epsilon x_1^2 + \epsilon d x_2^2 - d x_3^2$. By Lemma \ref{lem: isotropy criterion}, this is anisotropic if and only if $(\epsilon d, \epsilon) = (d,\epsilon) = 1$ if and only if $d$ is a uniformizer (up to squares). By Lemma \ref{lem: cubic characterization}, it is isotropic if and only if $F/K$ is Hecke unramified.     
\end{proof}

\subsection{Algebraic proof of Hecke reciprocity for cubic fields}

Now suppose $K$ is a number field, and let $F/K$ be an extension of degree $3$. The following is a special case of Theorem \ref{thm: hecke reciprocity}, but we will prove it in a different way.

\begin{theorem}\label{thm: reciprocity law cubics}
    Let $F$ be a cubic extension of $K$ and let $t \in \Sel_2^\un(F/K)$. Then the number of 
 Hecke  primes of $F/K$ that are inert in $F(\sqrt{t})/F$ is even. 

\end{theorem}

\begin{proof}
    Let $S$ be the set of primes of $K$ with splitting type $(1^21)$.  For each $v \in S$, write $v = w_1^2w_2$ and let $\sqrt{d} \in F_{w_1}$ be as in Proposition \ref{prop: partially ramified cubic computation}. 
 Choose a single $z \in F$ such that for each $v \in S$, we have $w_1(z - \sqrt{d}) \gg 1$ and $w_2(z) \gg 1$. Then $\beta := 3z - \Tr_{F/K}(z)$ satisfies $\Tr_{F/K}(\beta) = 0$.  If we write $z = (a+b\sqrt{d}, c) \in F \otimes K_v \simeq F_{w_1} \times F_{w_2}$, then 
    \[\beta = (3a+3b\sqrt{d} - 2a - c,3c - 2a -c) = (a -c + 3b\sqrt{d}, -2a-2c), \]
    and our choice of $z$ means that $v(a-c) \gg 1$ whereas $v(b) = 0$. 

    Let $q_\beta \colon H^1(K, M) \to \Br(K)[2]$ be the global refinement. For $t \in \Sel_2^\un(F/K)$, consider the element $\alpha = q_\beta(t)$. Let $\inv_v \colon \Br(K_v)[2] \simeq \Z/2\Z$ be the invariant map.  By Proposition \ref{prop: partially ramified cubic computation}, if $F/K$ is Hecke ramified at $v$ and $\mathrm{res}_v(t)$ is equal to the non-trivial class in $V_v \subset H^1(K_v, M)$, then $\inv_v(\mathrm{res}_v\alpha) = 1/2$. Otherwise, we have $\inv_v(\mathrm{res}_v(\alpha)) = 0$, by Lemmas \ref{lem: cubic characterization}-\ref{lem: field extensions}; if $v$ is archimedean, then this is because $t$ is locally trivial at $v$. Class field theory gives $\sum_v \inv_v(\mathrm{res}_v(\alpha)) =0$, so we see that the number of Hecke ramified primes $v$ of $K$ above which there is a prime which is inert in $F(\sqrt{t})$ is even. By Lemma \ref{lem: cubic characterization}, above each such $v = w_1^2w_2$ there is exactly one Hecke prime (namely, $w_1$), so the number of Hecke primes which are inert in $F(\sqrt{t})$ is also even.   
\end{proof}

\begin{remark}
 {\em 
 The proof of Theorem \ref{thm: reciprocity law cubics} did not use Hecke's Theorem \ref{thm: hecke}. In fact, it gives a new proof of Hecke's theorem for cubic $F/\Q$, since applying it to all $ \chi \in \Cl_{F/\Q}^\vee[2] \simeq \Sel_2^\un(F/\Q)$, we conclude that $\chi([H_{F/\Q}]) =0$ and hence $[H_{F/\Q}]$ is a  square in $\Cl_F$. It follows that $[\Diff_{F/\Q}]$ is a square in the class group as well. 
 }   
\end{remark}

\section{A canonical quadratic refinement for Kummer extensions}\label{sec: kummer}

In this section, we consider Kummer extensions $F = K(\sqrt[m]{n})$, where $m = 2g+1$ is an odd prime. 
 There is then a natural choice of $\beta \in F$ of trace $0$, namely $\beta = \sqrt[m]{n}$. Let $q \colon H^1(K, M) \to \Br(K)[2]$ be the associated quadratic refinement (previously called $q_\beta$). 

\begin{proposition}\label{prop: Kummer kernel condition}
A class $t \in H^1(K, M)$ lies in $\ker(q)$ if and only if the quadratic space $\frac{1}{m}\Tr_{F/K}(tx^2)$ $($of rank $m = 2g+1$ and discriminant $1)$ is split, i.e.\ contains a $g$-dimensional isotropic subspace. 
\end{proposition}
\begin{proof}
    This follows from Proposition \ref{prop: obstruction in terms of isotropy}, since $f'(\beta) = m\beta^{2g}$ is a scalar times a square. 
\end{proof}

The purpose of this section is to determine to what extent elements of $V$ (in the local case) and  $\Sel_2^\un(F/K)$ (in the global case) belong to $\ker(q)$, as this will be necessary information in the proofs of Theorems \ref{thm: over Q(omega)} and \ref{main thm: over Q} and for the heuristics in  \S\ref{sec: Gamma-extension heuristics}. 
%we do the analysis for general prime degree Kummer extensions. 

\subsection{Local Kummer extensions}

Let $K$ be a finite extension of $\Q_p$ with normalized valuation $v$. Write $\zeta_m$ for a primitive $m$-th root of unity, and let $d$ (resp.\ $e$) be the degree (resp.\ ramification index) of the extension $K(\zeta_m)/K$.  Let $F = K[x]/(x^m - n)$ be a Kummer $K$-algebra of odd prime degree. Without loss of generality, we assume that $0 \leq v(n) < m$. 
\begin{lemma}\label{lem: Kummer Hecke ramification}
  $F/K$ is Hecke ramified if and only if $p = m$, $F/K$ is tamely ramified, and $e$ is even.  In this case, $n$ is an $m$-th power in $K^\times$ and a prime of $\O_F$ is a Hecke prime if and only if it is ramified. The number of such Hecke primes is $(p-1)/d$. 
\end{lemma}

\begin{proof}
    If $p \neq m$, and $F/K$ is ramified, then $v(n) > 0$ and $F/K$ is a totally ramified field extension. It is Hecke unramified by Lemma \ref{lem: field extensions}.  Of course, if $F/K$ is unramified then it is also Hecke unramified.

    Now suppose $p = m$. Note that $e$ divides $d$ but may be strictly smaller.  First consider the case that $n$ is an $m$-th power. Then $F \simeq K \times K(\zeta_p)^{(p-1)/d}$. We have $\Disc_{F/K} = (\pi)^{(p-1)(e-1)/d}$ and $(\pi) = \m_1 \prod_{i = 1}^{(p-1)/d} \m_i^e$. So $\mathrm{Diff}_{F/K} = \prod_{i =1}^{(p-1)/d} \m_i^{e-1}$ and 
    \[H_{F/K} = \m_1^{(p-1)(e-1)/d}\prod_{i = 1}^{(p-1)/d} \m_i^{e(p-1)(e-1)/d - (e-1)} = \m_1^{(p-1)(e-1)/d}\prod_{i = 1}^{(p-1)/d} \m_i^{(e-1)(\frac{e(p-1)}{d} - 1)}.\]
    This is a square of an ideal if and only if $e$ is odd, as claimed.

    If $n$ is not an $m$-th power in $K^\times$, then $x^m - n$ is irreducible over $K$, hence $F/K$ is a field extension (of degree $m = p$). If $F/K$  is ramified, then it must be totally (and wildly) ramified, and hence Hecke unramified by Lemma \ref{lem: field extensions}. 
\end{proof}

Define the subspace $V \subset H^1(K, M)$ as in \S\ref{subsec: quadratic refinements over local fields}. 
 To determine $\ker(q) \cap V$, we will use the Hasse-Witt invariant $\mathrm{HW}(Q)$ of a quadratic form $Q$ over $K$. By definition 
 \[\mathrm{HW}(Q) = \prod_{i < j} (a_i, a_j) \in \{\pm 1\},\] 
 where $Q \sim \sum a_ix_i^2$ is any diagonal quadratic form equivalent to $Q$.

\begin{lemma}\label{lem: Hasse-Witt}
Let $L/K$ be an \'etale $K$-algebra of degree $2g+1$ and let $\Delta := \disc(\Tr_{F/K}(x^2)) \in F^\times/F^{\times 2}$ be the discriminant. Then the quadratic space $\Tr_{K/F}(tx^2)$ is split if and only if it has Hasse-Witt invariant $((-1)^{g(g+1)/2}N(t)^g\Delta^g,-1)$.
\end{lemma}
\begin{proof}
The rank, discriminant, and Hasse-Witt invariant form a complete set of invariants for equivalence classes of quadratic forms over the local field $K$ \cite{MR522835}. There is a unique split space of discriminant $d$ and rank $2g+1$ over $K$, namely $\langle (-1)^gd \rangle \oplus gH$ where $H$ is the hyperbolic plane. Denote its Hasse-Witt invariant by $\varepsilon_d = ((-1)^{g(g+1)/2}d^g,-1)$. Therefore, $\Tr_{F/K}(tx^2)$ splits if and only if it has Hasse-Witt invariant equal to $\varepsilon_{N(t)\Delta}$. 
\end{proof}

\begin{remark}\label{rem: split quadratic form}
{\em
    In our setting of Kummer extensions, the quadratic form $\frac{1}{m}\Tr_{F/K}(tx^2)$ has discriminant $1$. If $p > 2$, then it is split if and only if its Hasse-Witt invariant is $1$.  
}
\end{remark}

\begin{theorem}\label{thm: local Kummer obstruction}
    If $p \neq m$ or $e \in \{1,m-1\}$, then $V \subset \ker(q)$ if and only if $F/K$ is Hecke unramified.  
\end{theorem}

\begin{proof}
    When $p\neq m$ or when $F/K$ is either unramified or totally ramified, this follows from Lemma \ref{lem: Kummer Hecke ramification} and Lemmas \ref{lem: field extensions}-\ref{lem: locally unobstructed at unramified primes}. 
    
    So we may assume now that $p = m$ and moreover that $n \in K^{\times m}$. As before, we have $F \simeq K \times K(\zeta_m)^h$, where $d = [K(\zeta_m) \colon K]$ and $hd = m-1 = p-1$.  If $e = 1$, then $F/K$ is unramified, a case we have already covered. So it remains to consider the case $e = d = m-1 = p-1$, where we must show (by Lemma \ref{lem: Kummer Hecke ramification}) that $V \not\subset \ker(q)$. In this case, $F \simeq K \times K(\zeta_m)$ and the unique  non-zero element $t \in V$ is $(1,\epsilon)$. The quadratic form $Q_t(x) = \frac{1}{m}\Tr_{F/K}(tx^2)$ is equivalent to 
    \[\langle m\rangle  \perp m\epsilon \Tr_{K(\zeta_m)/K}(x^2).\]
    Note that $K(\zeta_m)/K$ is of the form $K(\sqrt[m-1]{b})$ for some $b \in K^\times$, since $\Q_p$ contains $\mu_{p-1}$. 
    \begin{lemma}\label{lem: cyclotomic kummer}
        We have $b \equiv -p \pmod{K^{\times2}}$
    \end{lemma}
    \begin{proof}
        We may assume $K = \Q_p$. The discriminant of $\Q(\zeta_p)/\Q$ is $(-1)^{(p-1)/2}p^{p-2} \equiv (-1)^{(p-1)/2} p$. On the other hand, the discriminant of $x^{p-1} - b$ is 
        \[(-1)^{(p-1)(p-2)/2}(p-1)^{p-1}(-b)^{p-2} \equiv (-1)^{(p-1)(p-2)/2 + 1}b\] 
        Equating these two, we find that $b \equiv -p$. 
    \end{proof} 
    
    A formula of Serre \cite[pg. 672]{SerreWittInvariant84} for the trace form of a Kummer extension gives 
    \[\Tr_{K(\zeta_m)/K}(x^2) \sim \langle p-1, (-1)^{\frac{p-1}{2}}(p-1)p\rangle \perp H^{(p-1)/2 - 1} \sim \langle -1,(-1)^{\frac{p+1}{2}}p\rangle \perp H^{(p-1)/2 - 1}.\]
    The Hasse-Witt invariant of $Q_t$ is then 
    \begin{align*}
    \mathrm{HW}(Q_t) &= (m,m)^{\binom{p}{2}}\mathrm{HW}(\langle -\epsilon, (-1)^{\frac{p+1}{2}}\epsilon p, \epsilon, -\epsilon,\cdots, \epsilon, -\epsilon\rangle)\\
    &=  (m,m)^{\frac{p-1}{2}} (\epsilon, \pm p)\mathrm{HW}(\langle -1,(-1)^{\frac{p+1}{2}}p,1,-1,\cdots, 1,-1\rangle)\\
    &=(p,p)^{\frac{p-1}{2}} (\epsilon, p)((-1)^{\frac{p+1}{2}}p,-1)^{\frac{p-1}{2}}\\
    &= (p,-1)^{\frac{p-1}{2}} (\epsilon, p)(p,-1)^{\frac{p-1}{2}}\\
    &= (\epsilon,p)\\
    &=\begin{cases}
        1 & \mbox{ if } v(p) \equiv 0\pmod{2}\\
        -1 & \mbox{ if } v(p) \equiv 1\pmod{2}
    \end{cases}.
    \end{align*}

    However, $v(p)$ is odd, by Lemma \ref{lem: cyclotomic kummer} and our assumption that $K(\zeta_m)/K$ is totally ramified of degree $m-1$.  Thus $\mathrm{HW}(Q_t) = -1$ and $Q_t$ is not split.  Hence $V \not\subset \ker(q)$, by Proposition \ref{prop: Kummer kernel condition} and Remark \ref{rem: split quadratic form}. 
\end{proof}

\subsection{Global Kummer extensions}
Now  let $K$ be a number field.  The following is a more precise version of Hecke reciprocity for prime degree pure fields $F = K(\sqrt[m]{n})$ satisfying a mild hypothesis. Note that by Lemma \ref{lem: Kummer Hecke ramification}, if $v$ is a prime of $K$ which is Hecke ramified in $F$, then $v$ lies above $m$ and is tamely ramified in $F$. Moreover, the Hecke primes of $F/K$ are exactly the ramified primes of $F$ lying over such a prime $v$ of $K$.  Denote by $q_v \colon H^1(K_v, M) \to \Z/2\Z$ the local quadratic refinement.

\begin{theorem}\label{thm: global Kummer reciprocity law}
    Assume that for all primes $v$ of $K$ above $m$, the ramification index $e(K_v(\zeta_m)/K_v)$ is $1$ or $m-1$.  Let $F = K(\sqrt[m]{n})$ be a degree $m$ Kummer extension of $K$ and let $t \in \Sel_2^\un(F/K)$. Then $q_v(t) = 1/2 \in \Br(K_v)[2]$ if and only if $v$ is tamely ramified in $F/K$ and $t \notin F_w^{\times 2}$, where $w$ is the unique ramified prime above $v$. Since $\sum_v q_v(t) = 0$, the number of Hecke primes $w$ of $F/K$ such that $w$ is inert in $F(\sqrt{t})/F$ is even.
\end{theorem}

\begin{proof}
By Lemma \ref{lem: Kummer Hecke ramification} and Theorem \ref{thm: local Kummer obstruction},  we have $q_v(t) = 0$ for all primes $v$ except possibly those $v$ dividing $m$ which split as $v = w_1w_2^{m-1}$. For the latter, the space $V_v \subset H^1(K_v, M)$ is $1$-dimensional and a class $t \in \Sel_2^\un(F/K)$ restricts to the non-trivial class if and only if $t \notin F_{w_2}^{\times 2}$. Moreover, in this case, we have $q_v(\mathrm{res}_v(t)) = 1/2$ by Theorem \ref{thm: local Kummer obstruction} and Lemma \ref{lem: Kummer Hecke ramification}. Since $\sum_v \inv_v(q_v(\mathrm{res}_v(t))) = 0$, we see that the number of such primes $v$ is even. Finally, note that the non-trivial class in $V_v$ projects to a square in $F_{w_1}^\times$ since the norm of the class $\epsilon \in F_{w_2}^\times/F^{\times 2}_{w_2}$ is trivial. So the unique ramified prime $w$ above such a $v$ is the unique Hecke prime above $v$. Thus, the number of Hecke primes $w$ in $F$ with $t \notin F_w^{\times 2}$ is also even.   
\end{proof}

The hypothesis in our theorem always holds when $K = \Q$, so we obtain:

\begin{corollary}\label{thm: reciprocity for Kummer extensions over Q}
    Let $F_n = \Q(\sqrt[m]{n})$ be a degree $m$ Kummer extension and let $t \in \Sel_2^\un(F_n)$. Then $t$ lies in the kernel of $q \colon \Sel_2^\un(F_n) \to \Br(\Q)[2]$ and is locally square at all primes of $F_n$ above $m$. 
\end{corollary}
\begin{proof}
    The unique prime $p$ of $\Q$ above $m$ (namely $p = m$ itself) satisfies $[\Q_p(\zeta_m) \colon \Q_p] = m-1$, so the hypotheses of Theorem \ref{thm: global Kummer reciprocity law} hold. Moreover, this is the only prime which could be Hecke ramified in $F_n$. If $F_n/\Q$ is wildly ramified at $p$, then $V_p = 0$ and every $t \in \Sel_2^\un(F_n)$ is locally a square at the unique prime $w$ of $F_n$ above $p$. If $F_n/\Q$ is tamely ramified at $p$, then $(p) = w_1 w_2^{p-1}$, and $w_2$ is the unique Hecke prime of $F_n/\Q$. It follows that $t$ is a square at $w_1$ (for local reasons) and a square at $w_2$ as well (because of the reciprocity law in Theorem \ref{thm: global Kummer reciprocity law}).  
\end{proof}

The hypothesis in Theorem \ref{thm: global Kummer reciprocity law} also always holds for cubic extensions. We highlight the following special case (which would be relevant for versions of Theorems \ref{thm: over Q(omega)} and \ref{main thm: over Q} over other number fields).
\begin{corollary}\label{cor: Kummer cubic global result}
    \label{thm: cubic Kummer brauer obstruction over number fields}
Suppose that $K \otimes_\Q \Q_3 \simeq \prod_{i = 1}^N L_i$, with at most one factor $L_i$ not containing $\zeta_3$.  Then for every cubic Kummer extension $F_n = K(\sqrt[3]{n})$, the classes $t \in \Sel_2^\un(F_n/K) \subset F_n^\times/F_n^{\times 2}$ all lie in the kernel of $q \colon \Sel_2^\un(F_n/K) \to \Br(K)[2]$ and are local squares at any Hecke prime of $F_n$. 
\end{corollary}
\begin{proof}
    As before, noting that if $L_i$ contains $\zeta_3$ then the corresponding prime of $K$ is Hecke unramified in $F/K$. (In particular, even though we say ``any Hecke prime'' in the statement of the corollary, there is actually at most one such prime under the given hypothesis.) 
\end{proof}

\section{Pairs of binary cubic forms with vanishing $A_1$-invariant}\label{sec: counting pairs of binary cubics}

We recall some facts from \cite[\S3-4]{AlpogeBhargavaShnidman}, which will be used in the next section. Let $K$ be a field of characteristic not $2,3$ and let $G = \SL_2^2/\mu_2$, the algebraic group over $K$. Let $V = K^2 \otimes \Sym^3(K^2)$ be the $8$-dimensional irreducible $G$-representation of pairs of binary cubic forms. The ring of invariants $\overline{K}[V]^{G(\overline{K})}$ is equal to $\overline{K}[A_1,A_3]$, where $A_i$ is an invariant of degrees $2i$. Explicitly, we have
\begin{equation}\label{A1formula}
    A_1((F_1,F_2)) = r_1 r_8 - 3r_2 r_7 + 3r_3 r_6 - r_4 r_5,
    \end{equation}
where $F_1(x,y) = \sum_{i=0}^3 \binom3ir_{i+1} x^{3-i} y^i$ and $F_2(x,y) = \sum_{i=0}^3 \binom3i r_{i+5} x^{3-i} y^i$. The invariant $A_3$ is given by the degree 3 invariant $J$ of the covariant binary quartic form $G = (\partial_xF_1)(\partial_yF_2) - (\partial_yF_1)(\partial_xF_2)$. More precisely, 
\begin{equation}\label{A3formula}A_3(v) =  \frac{1} {108}\Bigl(J(G(v))- A_1(v)^3\Bigr)\end{equation}
for $v = (F_1,F_2) \in V(\C)$.  Here, we have normalized the invariants $J$ and $A_3$ so that the pair of cubic forms $v_n := (3xy^2,x^3 + 2ny^3)$ satisfies $J(G(v_n)) = 108n$ and $A_3(v_n) = n$.

The locus $A_1 = 0$ forms a quadric $Y \subset V$. 
For each $n \in K^\times$, there is a distinguished $G(K)$-orbit in $V(K)$ with invariants $A_1 = 0,A_3 = n$, namely the orbit of $v_n$.  The stabilizer $\mathrm{Stab}(v_n) \subset G$ is of order $4$, and is isomorphic to the $K$-group scheme 
\[M_n = \ker\left(\Res_{F_n/K}(\mu_2) \stackrel{\Nm}{\longrightarrow} \mu_2\right),\]
where $F_n = K[x]/(x^3 - n)$.  Let $Y(K)_n \subset Y(K)$ be the subset of pairs with invariants $A_1 = 0,A_3 = n$. Let \[H^1(K, \mathrm{Stab}(v_n)) \stackrel{q_n}{\longrightarrow} H^1(K, G)\]
be the map induced by the inclusion $\mathrm{Stab}(v_n) \hookrightarrow G$. As explained in \cite{BhargavaGross2013}, $q_n$ can be identified with the map $H^1(K, M_n) \to \Br(K)[2]$ of Proposition \ref{prop: obstruction in terms of isotropy} (with $\beta = \sqrt[3]{n}$) which we will also denote by $q_n$. 
 By arithmetic invariant theory \cite[Proposition 1]{BhargavaGross2014}, we have:

\begin{proposition}\label{prop: parameterization}
There is a bijection, functorial in $K$, between $G(K)\backslash Y(K)_n$ and $\ker(q_n)$.
\end{proposition}

A $G(\Z)$-invariant function $\phi:Y({\Z})\to[0,1]\subset \R$ is said to be {\it defined by congruence conditions} if, for all primes $p$, there exist functions $\phi_p:Y({\Z_p})\to[0,1]$ satisfying the following conditions:
\begin{itemize}
\item[(1)] For all $y\in Y(\Z)$, the product $\prod_p\phi_p(y)$ converges to $\phi(y)$;
\item[(2)] Each function $\phi_p$ is 
locally constant outside a closed subset of $Y({\Z_p})$ of measure zero.
\end{itemize}
We say that such a function $\phi$ is {\it acceptable} if for
sufficiently large primes $p$, we have $\phi_{p}(y)=1$ whenever
$p^2 \nmid A_3(y)$.  Let $N_\phi(Y(\Z);X)$ denote the weighted number of irreducible  $G(\Z)$-orbits of elements $y\in Y(\Z)$ with $|A_3(y)|<X$, where the orbit of each such $y$ is weighted by $\phi(y)$. 

In the next section, we will take $K = \Q$ and use the following counting result for $G(\Z)$-orbits on $Y(\Z)$ with bounded $A_3$-invariant \cite[Theorem 4.1]{AlpogeBhargavaShnidman}.

\begin{theorem}\label{thm: counting in the quadric over Q}
  Let $\phi:Y(\Z)\to[0,1]$ be a $G(\Z)$-invariant acceptable function that is defined by the functions $\phi_{p}:Y({\Z_p})\to[0,1]$. Then  
\begin{equation}
N_\phi(Y(\Z);X)
  = 
  X  \cdot 
 \int_{\scriptstyle{y\in G(\Z) \backslash Y(\R)}\atop\scriptstyle{|A_3(y)|<1}}
dy \cdot \prod_{p}
  \int_{y\in Y({\Z_{p}})}\phi_{p}(y)\,dy\,+\,O_\phi\left(X^{1 - \delta}\right);
\end{equation}
  here $dy$ is the $G(\R)$-invariant $($resp.\ $G(\Z_p)$-invariant$)$ measure $dr_2\,dr_3\cdots dr_8/(\partial A_1/\partial r_1)$ on $Y(\R)$ $($resp.\ $Y(\Z_p))$,
  where $r_1,\ldots,r_8$ are the coordinates on $V$.
\end{theorem}\noindent

The proof of Theorem \ref{thm: counting in the quadric over Q} could presumably be adapted to prove an analogous result over any number field $K$, using versions of the circle method, the divisor bound, and the geometric sieve for quadrics over $K$, but we have not worked through these details. However, if $K$ is an imaginary quadratic field with class number $1$, we have:

\begin{theorem}\label{thm: counting in the quadric over imaginary quadratic}
  Suppose $K$ is an imaginary quadratic field with class number $1$. Let $\phi:Y(\O_K)\to[0,1]$ be a $G(\O_K)$-invariant acceptable function defined by functions $\phi_{p}:Y({\O_{K,p}})\to[0,1]$. Then  
\begin{equation}
N_\phi(Y(\O_K);X)
  = 
  X  \cdot 
 \int_{\scriptstyle{y\in G(\O_K) \backslash Y(\C)}\atop\scriptstyle{|A_3(y)|<1}}
dy \cdot \prod_{p}
  \int_{y\in Y({\O_{K,p}})}\phi_{p}(y)\,dy\,+\,O_\phi\left(X^{1 - \delta}\right);
\end{equation}
  here $dy$ is the $G(\C)$-invariant $($resp.\ $G(\O_{K,p})$-invariant$)$ measure $dr_2\,dr_3\cdots dr_8/(\partial A_1/\partial r_1)$ on $Y(\C)$ $($resp.\ $Y(K_p))$,
  where $r_1,\ldots,r_8$ are the coordinates on $V$.
\end{theorem}
\begin{proof}
As $\O_K$ is a PID with a finite group of units, the proof of \cite[Theorem 4.1]{AlpogeBhargavaShnidman} goes through essentially unchanged, using the circle method over $\O_K$ as in \cite{Browningcubicforms}. To carry this out, one must use the analogue of \cite[Theorem 1.1]{BrowningHeathBrown} over $\O_K$, but again the proof is essentially the same.      
\end{proof}

\section{$2$-torsion in the class groups of  $\Q(\sqrt[3]{d})$}

We now specialize \S\ref{sec: counting pairs of binary cubics} to the case $K = \Q$. For any $n \in \Q^\times$ and any prime $p \leq \infty$, write 
\[V_{n,p} \subset H^1(\Q_p,M_n) \subset (F_n \otimes \Q_p)^\times/(F_n \otimes \Q_p)^{\times 2}\] 
for the subspace of local conditions defining $\Sel_2^\un(F_n/\Q) = \Sel_2^\un(F_n)$, as in \S\ref{subsec: relative signature groups}. Write
\[q_{n,p} \colon H^1(\Q_p, M_n) \to \Z/2\Z\] 
for the local quadratic refinement of \S\ref{sec: kummer}, and write $q_n \colon H^1(\Q, M_n) \to \Br(\Q)[2]$ for the global one. 

\subsection{The wild family}
Consider first the wild fields $F_n$, i.e.\ those with $n \in \Z$ cube-free and $n \not \equiv \pm 1\pmod{9}$.  In order to apply Theorem \ref{thm: counting in the quadric over Q}, we use: 

\begin{proposition}\label{prop: local signature elements are unobstructed}
Suppose $n \not\equiv \pm 1 \pmod{9}$. Then $V_{n,p} \subset \ker(q_{p,n})$, for every prime $p$.
\end{proposition}
\begin{proof}
This follows from Lemma \ref{lem: Kummer Hecke ramification} and Theorem \ref{thm: local Kummer obstruction}.  
\end{proof}

We say  $v \in Y(\Q)_n$ is {\it unramified} if, for every prime $p$, the $G(\Q_p)$-orbit of $v$ corresponds to an element of $V_{n,p}$ under the bijection of Proposition \ref{prop: parameterization}. Let $Y(\Q)_n^\un$ denote the set of unramified vectors with invariant $n$. 

\begin{corollary}\label{cor: Z bijection wild}
For each $n \not\equiv \pm 1\pmod{9}$, there exists a bijection $\Sel_2^\un(F_n) \to G(\Q)\backslash Y(\Q)^\un_n$. Moreover, there exists $N \geq 1$,  independent of $n$, such that each $G(\Q)$-orbit in $Y(\Q)^\un_n$ is represented by an element in $\frac{1}{N}V(\Z)$. 
\end{corollary}

\begin{proof}
    By Proposition \ref{prop: local signature elements are unobstructed}, we have $\Sel_2^\un(F_n) \subset \ker(q_n)$, and so the first sentence follows from Proposition \ref{prop: parameterization}.  The proof of integrality (i.e.\ that the $G(\Q)$-orbits have representatives with bounded denominators) is exactly as in \cite[Theorem 3.19]{AlpogeBhargavaShnidman}. Integrality can also be deduced from the latter result (without reproving it) since for $p > 3$, the local conditions at $p$ defining $\Sel_2^\un(F/\Q)$ agree with the local conditions defining $\Sel_2(E_n)$, where $E_n \colon y^2 = x^3 + n^2$. Indeed, if $p$ divides $n$ then $H^1(\Q_p, M_n) = 0$ and if $p$ does not divide $n$, then the local condition is the unramified classes, in both cases.  
\end{proof}

\begin{theorem}\label{thm: wild selmer avg}
    Ordering the wild fields $F_n = \Q(\sqrt[3]{n})$ by $|n|$, we have $\avg_n\, |\Cl_{F_n}[2]| = 2$.
\end{theorem}

\begin{proof}
 Let $U \subset \Z$ be the subset of cubefree integers $n \not\equiv \pm 1 \pmod{9}$. For each prime $p$, let $U_p$ be the closure of $U$ in $\Q_p$ (so that $U_p = \Z_p$ for $p \nmid 3\infty$). For $v \in Y(\Z)$, we let $m(v)$ be the number of $G(\Z)$-orbits in $G(\Q)v$, and set $\phi(v) = 1/m(v)$.  By Corollary \ref{cor: Z bijection wild}, we may compute $\avg_n|\Sel_2^\un(F_n)|$ by estimating, for each $X$, the number of $\phi$-weighted unramified $G(\Z)$-orbits $G(\Z)v$ with $|A_3(v)| < X$, and then taking the limit as $X \to \infty$.   Theorem \ref{thm: counting in the quadric over Q} gives an Euler product expression for the main term. Using the change of variable formula \cite[Proposition 4.10]{AlpogeBhargavaShnidman} and standard manipulations \cite[proof of Theorem 5.3]{AlpogeBhargavaShnidman}, the Euler factor at $p$ is equal to 
 \[\nu_p = \int_{n \in U_p} \frac{|V_{n,p}|}{|M_n(\Q_p)|}dn,\]
i.e.\ the average over $n \in U_p$ of the number of unramified $G(\Q_p)$-orbits on $Y(\Q_p)$ with invariant $n$, weighted by the size of their stabilizers. 
Thus, we arrive at the formula
\[\avg_{n \in U} |\Sel_2^\un(F_n)| = 1 + 2\prod_p \nu_p,\] 
where the factor of $2$ is the Tamagawa number of $G$. For $p \neq \infty$, we have $\nu_p  = 1$ by Lemma \ref{lem: local selmer ratio}. For $p = \infty$, we have $\nu_\infty = 1/2$.  Putting this together, we obtain,
\[\avg_n |\Sel_2^\un(F_n)| =1 + 2\cdot \frac12 \cdot 1 =  2.\]
The result now follows from Theorem \ref{thm: canonical isomorphism}.
\end{proof}

\subsection{The tame family}
Next we consider the tame fields $F_n = \Q(\sqrt[3]{n})$, i.e.\ those such that $n \equiv \pm 1\pmod{9}$. The wrinkle  here is that the analogue of Proposition \ref{prop: local signature elements are unobstructed} fails:

\begin{proposition}\label{prop: tame failure II}
Suppose $n \equiv \pm 1\pmod{9}$. Then the non-trivial element of $V_{n,3} \subset H^1(\Q_3, M_n)$ is not contained in $\ker(q_{n,3})$. 
\end{proposition}
\begin{proof}
This follows again from Lemma \ref{lem: Kummer Hecke ramification} and Theorem \ref{thm: local Kummer obstruction}
\end{proof}

Nonetheless, we still have the following global result.

\begin{proposition}\label{prop: selmer elements are unobstructed}
Each class $t \in \Sel_2^\un(F_n)$ lies in $\ker(q_n)$. 
\end{proposition}
\begin{proof}
This is a special case of the ``explicit'' version of Hecke reciprocity in Theorem \ref{thm: reciprocity for Kummer extensions over Q}.
\end{proof}

We define $Y(\Q)_n^\un$ as before and conclude similarly (allowing now $n \equiv \pm1 \pmod9$):

\begin{corollary}\label{cor: Z bijection}
For each $n \geq 1$, there exists a bijection $\Sel_2^\un(F_n) \to G(\Q)\backslash Y(\Q)^\un_n$. Moreover, there exists $N \geq 1$,  independent of $n$, such that each $G(\Q)$-orbit in $Y(\Q)^\un_n$ is represented by an element in $\frac{1}{N}V(\Z)$. 
\end{corollary}

However, because of the obstruction coming from Theorem \ref{thm: reciprocity for Kummer extensions over Q}, the computation of the moment of $|\Sel_2^\un(F_n)| = |\Cl_{F_n}[2]|$,  for the tame fields $F_n$, takes a different shape:

\begin{theorem}\label{thm: tame selmer avg}
    Ordering the tame fields $F_n = \Q(\sqrt[3]{n})$ by $|n|$, we have $\avg_n\, |\Cl_{F_n}[2]|= 1 + \frac12 $.
\end{theorem}

\begin{proof}
 Let $U \subset \Z$ be the integers congruent to $\pm 1\pmod9$, and for each prime $p$, let $U_p$ be the closure of $U$ in $\Q_p$. Proceeding as in the proof of Theorem \ref{thm: wild selmer avg}, and using Corollary \ref{cor: Z bijection}, we obtain
\[\avg_d \#\Sel_2^\un(F_n) = 1 + 2\prod_p \nu_p,\] 
where this time 
\[\nu_p = \int_{n \in U_p} \frac{\#V_{n,p} \cap \ker(q_{n,p})}{\#M_n(\Q_p)}dn.\]
The computation for $p \neq 3$ is as before. For $p = 3$ we have:

\begin{lemma}
$\nu_3 = \frac12$.
\end{lemma}
\begin{proof}
Since $F_n$ is tame, the prime $3$ has splitting type $(1^21)$ and hence $\#V_{n,3} = \#M_n(\Q_3) = 2$ by Lemma \ref{lem: local selmer ratio}. On the other hand, we have $\#V_{n,3} \cap \ker(q_{n,3})= 1$ by Proposition \ref{prop: tame failure II}. 
\end{proof}

Putting this together, we obtain,
\[\avg_n \#\Sel_2^\un(F_n) = 1 + 2\prod_p \nu_p =1 + 2\cdot \frac12 \cdot \frac12 \cdot \prod_{p \nmid 3\infty}  1 =  1 + \frac12,\]
as desired.
\end{proof}

\begin{proof}[Proof of Theorem \ref{main thm: over Q}]
This is Theorems \ref{thm: tame selmer avg} and \ref{thm: wild selmer avg}. 
\end{proof}

\section{Generalization to imaginary quadratic fields} \label{sec: imaginary quadratic fields}
Let $K$ be an imaginary quadratic field. For simplicity, we assume that $\Cl_K = 0$. Consider the cubic extensions $F_n = K(\sqrt[3]{n})$, with $n \in \O_K$.  For any prime $p \leq \infty$ of $K$, write 
\[V_{n,p} \subset H^1(K_p,M_n) \subset (F_n \otimes K_p)^\times/(F_n \otimes K_p)^{\times 2}\] 
for the subspace of local conditions defining $\Sel_2^\un(F_n/K) = \Sel_2^\un(F_n)$, as in \S\ref{subsec: relative signature groups}. Let 
\[q_{n,p} \colon H^1(K_p, M_n) \to \{0,1\}\] 
be the local obstruction map. If $p \nmid 3$, then $V_{p,n} \subset \ker(q_n)$ for all $n$, by Lemma \ref{lem: Kummer Hecke ramification} and Theorem \ref{thm: local Kummer obstruction}. When $p \mid 3$, the behavior depends on the local $3$-adic behavior of $K$, as we shall see.

\subsection{$K$ ramified at $3$} 
Suppose that $K$ is ramified at $3$, so that $(3) = \p^2$. (Since we assume $\Cl_K = 0$, this in fact forces $K = \Q(\sqrt{-3})$.)  In this case, we have:
\begin{proposition}\label{prop: local signature elements are unobstructed over K}
For every prime $p$ of $K$, we have $V_{n,p} \subset \ker(q_{n,p})$.
\end{proposition}
\begin{proof}
    This follows from Theorem \ref{thm: local Kummer obstruction}, once we check that there are no Hecke primes for the fields $F_n$. By Lemma \ref{lem: Kummer Hecke ramification}, the only possible Hecke prime is $p  = \p$, but since $e$ is odd in this case, even $\p$ is not a Hecke prime. 
\end{proof}

Let $Y(K)_n^\un$ denote the vectors $v \in Y(K)_n$ such that, for every prime $p$ of $K$, the $G(K_p)$-orbit of $v$ corresponds to an element of $V_{n,p}$ under the bijection of Proposition \ref{prop: parameterization}. 

\begin{corollary}\label{cor: Z[omega] bijection}
For each $n \in \O_K$, there exists a bijection $\Sel_2^\un(F_n/K) \to G(K)\backslash Y(K)^\un_n$. Moreover, there exists $N \geq 1$,  independent of $n$, such that each $G(K)$-orbit in $Y(K)^\un_n$ is represented by an element in $\frac{1}{N}V(\O_K)$. 
\end{corollary}

\begin{proof}
As before.
\end{proof}

The following is Theorem \ref{thm: over Q(omega)} from the introduction.

\begin{theorem}\label{thm: C_3 selmer avg}
    Let $K = \Q(\sqrt{-3})$. Ordering the extensions $F_n = K(\sqrt[3]{n})$, for $n \in \O_K$,  by $\Nm(n)$, we have $\avg_n \,  \#\Sel_2^\un(F_n/K) = 3/2$.
\end{theorem}

\begin{proof}
Using Theorem \ref{thm: counting in the quadric over imaginary quadratic}, and arguing as before, we have
\[\avg_d \#\Sel_2^\un(F_n/K) = 1 + 2\prod_p \nu_p,\] 
where 
\[\nu_p = \int_{n \in \O_{K_p}} \frac{\#V_{n,p}}{\#M_n(K_p)}dn,\]
If $p \neq \infty$ we have $\nu_p = 1$ by Proposition \ref{lem: local selmer ratio}. 
On the other hand, we have $\nu_\infty = \#V_{n,\infty}/\#M_n(\C) =1/4$.
Putting this together, we obtain,
\[\avg_n \#\Sel_2^\un(F_n) = 1 + 2\prod_p \nu_p =1 + 2\cdot \frac14 =  1 + 1/2,\]
proving the theorem.
\end{proof}

\subsection{$K$ inert at $3$}
Suppose that $K$ is inert at $3$, so that $(3) = \p$ is the unique prime above $3$.  
\begin{proposition}\label{prop: tame failure III}
Let $p$ be a prime of $K$ and $n \in K^\times$. If $p \neq \p$ or if $n$ is not a cube in $K_\p$, then $V_{n,p} \subset \ker(q_n)$. If $p = \p$ and $n$ is a cube in $K^\times_\p$, then the non-trivial element of $V_{n,\p} \subset H^1(K_\p, M_n)$ is not contained in $\ker(q_{n,\p})$. 
\end{proposition}
\begin{proof}
Again from Lemma \ref{lem: Kummer Hecke ramification} and Theorem \ref{thm: local Kummer obstruction}
\end{proof}

Despite the second sentence of Proposition \ref{prop: tame failure III}, for every $n \in K^\times$, the group $\Sel_2^\un(F_n/K)$ lies in $\ker(q_n)$. Indeed, this follows from Corollary \ref{thm: reciprocity for Kummer extensions over Q}, as in the tame case over $\Q$. Let $Y(K)_n^\un$ denote the vectors $v \in Y(K)_n$ such that for every prime $p$ of $K$, the $G(K_p)$-orbit of $v$ corresponds to an element of $V_{n,p}$ under the bijection of Proposition \ref{prop: parameterization}. As before, we have: 

\begin{corollary}\label{cor: Z[i] bijection}
For each $n \in \O_K$, there exists a bijection $\Sel_2^\un(F_n/K) \to G(K)\backslash Y(K)^\un_n$. Moreover, there exists $N \geq 1$,  independent of $n$, such that each $G(K)$-orbit in $Y(K)^\un_n$ is represented by an element in $\frac{1}{N}V(\O_K)$. 
\end{corollary}

As over $\Q$, say that $F_n$ is wild if it is wildly ramified over $\p$. Otherwise, we say $F_n$ is tame.

\begin{theorem}\label{thm: average over Q(i)}
    Let $K = \Q(\sqrt{-d})$ for $d \in \{-1,-7,-19,-43,-67,-163\}$, and order the fields $F_n = K(\sqrt[3]{n})$, for $n \in \O_K$,  by $\Nm(n)$. Then as $F_n$ varies through wild $($resp.\ tame$)$ $F_n$, the first moment $\avg_n \,  \#\Sel_2^\un(F_n/K)$ is equal to $3/2$ $($resp.\ $5/4)$.
\end{theorem}

\begin{proof}
As before, we use Corollary \ref{cor: Z[i] bijection} and Theorem \ref{thm: counting in the quadric over imaginary quadratic}, to arrive at an Euler product expression
\[\avg_d \#\Sel_2^\un(F_n/K) = 1 + 2\prod_p \nu_p.\] 

If $p \nmid 3\infty$ we have $\nu_p = 1$ as before.  
We have $\nu_\infty = \#V_{n,\infty}/\#M_n(\C) =1/4$. Finally, just as in the wild case over $\Q$, we have $\nu_3 = 1$, whereas in the tame  case we have $\nu_3 = 1/2$. Putting this together, we obtain in the wild case:
\[\avg_n \#\Sel_2^\un(F_n/K) = 1 + 2\prod_p \nu_p =1 + 2\cdot \frac14 \cdot 1=  1 + 1/2,\]
whereas in the tame case we obtain
\[\avg_n \#\Sel_2^\un(F_n/K) = 1 + 2\prod_p \nu_p =1 + 2\cdot \frac14 \cdot \frac12=  1 + 1/4 = 5/4,\]
as claimed.
\end{proof}

\section{Heuristics for $\Cl_{F/K}[2]$ in families of $G$-extensions} \label{sec: random model for Cl[2]}

In \cite{Malle2010}, Malle gives heuristics for the distribution of the groups $\Cl_{F/K}[2]$ in families of odd degree extensions $F/K$ with fixed Galois group and unit group rank.  Below we attempt to reconcile his heuristics with Hecke reciprocity. For simplicity, we assume $K = \Q$. 

 Recall the notation from Section \ref{subsec: relative signature groups}. By definition, we have 
 \[\Sel_2^\un(F) = \Sel_2^\un(F/\Q) = \ker\left(W \to \prod_p W_p/V_p\right) \subset  W = H^1(\Q,M).\] 
 The isomorphism $\Sel_2^\un(F) \simeq \Cl_F^\vee[2]$ of Theorem \ref{thm: canonical isomorphism} allows us to view $\Cl_F^\vee[2]$ as a kernel of a matrix, as follows. For any finite set of finite primes $S$ of $\Q$, let $H^1_S \subset W$ be the subspace of classes that are unramified outside $S$. Let $Y_S = \oplus_{p \in S \cup \{\infty\}} W_p/V_p$.  Then $\Sel_2^\un(F) = \ker(H^1_S \to Y_S)$, realizing $\Sel_2^\un(F)$ as the kernel of a linear map between finite dimensional $\F_2$-vector spaces.

We show next that for large and generic enough choices of $S$, this matrix is nearly square.  Write $u = r_1 + r_2 - 1$ for the unit group rank and write $h = \dim \Cl_{F}[2]$. 
\begin{lemma}
    Let $T$ be the set of primes of $F$ above $S$. Let $\langle T\rangle$ denote the subspace of $\Cl_F/2\Cl_F$ spanned by the prime ideals in $T$. Then $\dim H^1_S = u + h + |T| - |S| - \dim \langle T \rangle$. 
\end{lemma}
\begin{proof}
When $S$ is empty, this follows from the relative version of (\ref{seq: selmer group ses}). We prove the general case by induction on $|S|$. Suppose the lemma is proven for $S$ and let $S' = S\cup \{p\}$. Let $w_1,\ldots, w_k$ be the primes of $F$ above $p$, ordered so that $F_{w_1}/\Q_p$ has odd degree. Let $T'$ be the primes of $F$ above $S'$.  It is enough to show that $\dim H^1_{S'} - \dim H^1_S = k-1 - n$ where $n = \dim\langle T'\rangle - \dim \langle T \rangle$. For each $1 \leq j \leq k$, let $H^1_S \subset H^1_j \subset H^1_{S'}$ be the subgroup of $H^1_{S'}$ consisting of square-classes whose odd valuations are contained in the set $T_j:= S \cup \{w_1,\ldots, w_j\}$. Consider the homomorphism $\psi \colon H^1_j \to \Z/2\Z$ sending a squareclass $t$ to $w_j(t) \pmod{2}$. It has kernel $H^1_{j-1}$.
When $j = 1$, we have $H^1_S = H^1_1$ since any $t \in F^\times$ with $w_1(t) \equiv 1\pmod{2}$ and $w_i(t) \equiv 0\pmod{2}$ for $i > 1$ does not have square norm. Thus, it is enough to show that for $j > 1$ we have $H^1_{j-1} \neq H^1_j$ if and only if the class of $w_j$ in $\langle T_j\rangle$ lies in $\langle T_{j-1}\rangle$. But $H^1_j \neq  H^1_{j-1}$ if and only if \[w_j = (t)I^2\prod_{w \in J}w\] 
for some subset $J \subset T_{j-1}$, some ideal $I$ and some $t \in F^\times$ with square norm if and only if the image of $w_j$ in $\langle T' \rangle$ lies in $\langle T \rangle +\langle w_1,\ldots, w_{j-1}\rangle = \langle T_j\rangle$, as desired. 

\end{proof}

\begin{lemma}
With notation as above, if $S$ contains $2$ then $\dim Y_S = 2u + |T| - |S|$.    
\end{lemma}
\begin{proof}
    First note that $\dim Y_\emptyset = r_1 - 1$. Write $m = 2g + 1$.  For a prime $p$, let $k_p$ be the number of primes of $F$ above $p$. Since $M$ is self-dual, the local Euler characteristic formula gives \[\dim W_p = \dim H^1(\Q_p,M) = \begin{cases}
        2k_p - 2  & \mbox{ if } p \neq 2\\
        2g + 2k_p-2 & \mbox{  if } p = 2
    \end{cases}\]
    For all primes $p$ we have $\dim V_p = k_p - 1$. It follows that for $S = \{2\}$ we have
    \[ \dim Y_{\{2\}} = r_1 - 1 + 2(g + k_2 - 1) - (k_2-1) = 2u +k_2 - 1,\]
    using that $m = 2g + 1 = r_1 + 2r_2$.  Similarly, if $2 \in S$ then $\dim Y_S = 2u + |T| - |S|$.  
\end{proof}

\begin{corollary}\label{cor: square}
    If $T$ contains a generating set for $\Cl_F/2\Cl_F$, then $\dim Y_S = \dim H^1_S + u$. 
\end{corollary}
\begin{proof}
    We have $\dim \langle T\rangle = h$, so this follows from the previous two lemmas. 
\end{proof}

Let $\mathcal{F}$ be the family of odd degree $m$ extensions $F/\Q$ with Galois group $S_m$ and unit rank $u = r_1 + r_2 - 1$. Corollary \ref{cor: square} shows that (by assuming $S$ contains $2$ and a generating set for $\Cl_F/2\Cl_F$) we may view $\Cl_F^\vee[2]$ as the kernel of a matrix with $N$ rows and $N + u$ columns, for some integer $N$. However, this matrix is not quite random, since we should account for Hecke reciprocity.  For this, suppose that $S$ contains the set $S_0$ of all primes that are Hecke ramified in $F$, and let $T_0 \subset T$ be the subset of Hecke primes of $F$. Consider the linear functional $f \colon H^1_S \to \Z/2\Z$ sending $t \mapsto \sum_{t \in T_0} (t,\pi_w)_w$ for any choice of uniformizers $\pi_w$, for $w \in T_0$. Then Hecke reciprocity (Theorem \ref{thm: hecke reciprocity}) is equivalent to the statement $\Sel_2^\un(F) \subset \ker(f)$. Thus $\Cl_F^\vee[2] = \ker(\ker(f) \to Y_S)$ is the kernel of a random $(N+u) \times (N-1)$ matrix over $\F_2$. (Since the average number of Hecke primes of $F/\Q$ is unbounded, the functional $f$ is non-zero with probability $1$, hence $\dim \ker(f) = N-1$ with probability $1$.) Dualizing and replacing $N$ with $N-1$, we see that $\Cl_F[2]$ is the cokernel of a square matrix over $\F_2$ modulo $u+1$ random relations.  

The final bit of structure to account for is the alternating bilinear pairing on $\Cl_F^\vee[2]$ given by
\[(\chi,\chi') = \chi(\psi_F(\chi')) - \chi'(\psi_F(\chi))\]
where $\psi_F$ is the isomorphism of (\ref{eq: LST map}). (In fact, there is a second alternating pairing $\omega$ on $\Cl_F^\vee[2]$ defined by \cite{LipnowskiSawinTsimerman}; all we will use is the existence of one such pairing.) 

%\artane{That the left hand side gives zero is Weil reciprocity. Indeed, $f(\text{div} g)/ g(\text{div} f) =  \prod_{v \in S} (g,f)_v$ where $S$ is the set of bad places dividing $\#\mu$, the order of the roots of unity. This vanishes as long as $g,f$ are chosen co-prime to each other and to $S$. Oops, forgot to divide by $2$, will fix soon...}

\begin{lemma}\label{lem: alternating}
    If a finite abelian $p$-group admits a non-degenerate alternating pairing, it can be realized as the cokernel of some large enough alternating matrix.
\end{lemma}
\begin{proof}
    In fact, each $p$-group with a non-degenerate alternating pairing has a positive probability of being realized as such a cokernel \cite[Theorem 3.9]{BhargavaKaneLenstraPoonenRains}.
\end{proof}

Lemma \ref{lem: alternating} and the discussion preceding it suggest that we should model $\Cl^\vee_F[2]$ as the cokernel of a large alternating matrix over $\F_2$ modulo $u+1$ random relations. 

\begin{lemma}\label{lem: alternating distribution}
    Let $G$ be the cokernel of a random alternating $n \times n$ matrix over $\F_2$. As $n \to \infty$, 
\begin{equation} \label{eqn: symmetric moments}
    \mathbb{E}(|\mathrm{Surj}(G,V)|) = |V|| \wedge^2 V|.
\end{equation}
\end{lemma}
\begin{proof}
    The cokernels of random alternating matrices over $\Z_2$ have moments given by $\lvert \Sym^2 H \rvert$ from \cite{BhargavaKaneLenstraPoonenRains} and in particular \cite[\S 4.1]{NguyenWood}. Let $\overline{X}$ be any $\F_2$-matrix. Since $\coker(\overline{X}) = \coker(X)/2\coker(X)$ for any lift $X$ of $\overline{X}$ to a matrix with coefficients in $\Z_2$ \cite[Lemma 2.25.]{DasZhangShiqiao}, we see that the cokernels of random alternating $\F_2$-matrices have the same $V$-moments as the cokernels of random $\Z_2$ and this moment is thus equal to $\lvert \Sym^2 V \rvert = \lvert V \rvert \lvert \wedge^2 V \rvert$ as desired. 
\end{proof}

To obtain our model for $\Cl_{F/K}[2]$ we should therefore take the distribution in Lemma \ref{lem: alternating distribution} and further mod out by $u+1$ random relations.  By \cite[proof of Lemma 8.12.]{LipnowskiSawinTsimerman}, we arrive at the distribution:
\begin{equation} \label{eqn: Malle}
    \mathbb{E}(|\mathrm{Surj}(\Cl_{F/K}[2],V)|) = \dfrac{| \wedge^2 V|}{|V|^u},
\end{equation}
which is the distribution predicted by Malle in the family of degree $m$ extensions $F/\Q$ with Galois group $S_m$ and unit rank $u = r_1 + r_2 - 1$. Thus, we have recovered Malle's heuristic via Hecke reciprocity. 

The following proposition records some of the properties of Malle's distribution.

\begin{proposition} \label{prop: main distribution}
    For every $u \ge 0$, there exists a unique probability distribution on isomorphism classes of $\F_2$-vector spaces having moments (\ref{eqn: Malle}). Explicitly, this distribution is given by: 
    $$
    \P(V) = \frac{\lvert \wedge^2 V \rvert}{\lvert \# \Aut_{\F_2} V \rvert \lvert V^{u} \rvert} \prod_{i \ge 1} (1+2^{-u-i})^{-1}
    $$
    and the $n^{th}$ moment of $\lvert V \rvert$ with respect to this distribution is given by: $$\mathbb{E}(\lvert V \rvert^n) = \prod_{k=0}^{n-1} (1+2^{k-u}).$$

\end{proposition}
\begin{proof}
The first sentence and the formula for $\P(V)$ follows from Lemma 6.1. of \cite{SawinWood}.

To prove the statement concerning the $n^{th}$ classical moment, we consider the functor $A \mapsto A/2A$ from  the category $\mathrm{FinAb}_2$ of finite abelian $2$-groups to the category $\mathrm{Vect}_{\F_2}$ of finite dimensional $\F_2$-vector spaces. Any surjection from $A \twoheadrightarrow V$ descends to a surjection $A/2A \twoheadrightarrow V$, so if $\mu$ is a distribution on $\mathrm{Vec}_{\F_2}$ determined by its moments, then any distribution $\nu$ on on $\mathrm{FinAb}_2$ whose $V$-moment matches those of $\mu$ for all $V \in \mathrm{Vect}_{\F_2}$ must pushforward to $\mu$ along $\mathrm{FinAb}_2 \rightarrow \mathrm{Vect}_{\F_2}$.

Malle's distribution on finite abelian $2$-groups in \cite[Conjecture 2.1]{Malle2010} has $V$-moments equal to (\ref{eqn: Malle}) by \cite[Lemma 12.5.]{SawinWood}. It follows that the distribution associated to (\ref{eqn: Malle}) is the pushforward of the distribution of \cite[Conjecture 2.1.]{Malle2010} on finite abelian $2$-groups under $\mathrm{FinAb}_2 \rightarrow \mathrm{Vect}_{\F_2}$. Since $A[2] \simeq A/2A$ for any finite abelian $2$-group $A$,  the $n^{th}$ moment of $2^{\rk_2 V}$ over the distribution \cite[Conjecture 2.1.]{Malle2010} coincides the with $n^{th}$ moment of $\lvert V \rvert$ the distribution (\ref{eqn: Malle}) and by \cite[Proposition 2.2.]{Malle2010} the former is equal to $\mathbb{E}(\lvert V \rvert^n) = \prod_{k=0}^{n-1} (1+2^{k-u})$. 
\end{proof}

Malle gives a heuristic for more general families $\mathcal{F}$, where the base field $K$ need not be $\Q$ and where the fixed Galois group $G = G_{F/K} \subset S_m$  of $F/K$ need not be all of $S_m$; see also \cite{AdamMalle}. Suppose $G \subset S_m$ is {\it absolutely irreducible} in the sense that the permutation $G$-representation $\chi$ for the index $m$ subgroup $\Gal(\tilde F/F)$ is such that $\varphi:= \chi - 1_G$ is irreducible over $\C$.  In this case, the heuristic distribution is the same as above, but with $u$ defined to be $u = \langle \chi_E, \varphi\rangle $, where $\chi_E$ is the character of the $G$-representation $\O_{\tilde F}^\times \otimes \Q$; see \cite[p.\ 466]{Malle2010}.  We expect that whenever $G$ has even order, the average number of Hecke primes for extensions $F/K$ in the family $\mathcal{F}$ is unbounded, and so the above reasoning should again apply to recover Malle's heuristic.

\section{Heuristics for $\Cl_{F/K}[2]$ in families of  $\Gamma$-extensions}\label{sec: Gamma-extension heuristics}

Let $F/K$ be a finite extension of number fields, and let $G_{F/K} = \Gal(\tilde F/K)$ be its Galois group. We may view $G_{F/K}$ as  $K$-group via the natural conjugation action $\Gal_K \to \Aut(G_{F/K})$.  Note that $G_{F/K}$ has trivial $\Gal_K$-action if and only if $F/K$ is Galois. 

Let $\Gamma$ be a $K$-group with underlying group $\Gamma_0$.
\begin{definition}
{\em     A {\it $\Gamma$-extension} is a finite extension $F/K$ together with an inclusion $\Gamma \hookrightarrow G_{F/K}$ of $K$-groups.}
\end{definition}  
Let $\mathcal{F}_\Gamma$ be the family of $\Gamma$-extensions of degree $|\Gamma|$, and let $F/K \in \mathcal{F}_\Gamma$.  We assume that the order $m = |\Gamma|$ is odd, and hence the extensions $F/K \in \mathcal{F}_\Gamma$ are of odd degree. 
 If $\Gamma$ is the constant $K$-group $\underline{\Gamma_0}$ associated to $\Gamma_0$, then $\mathcal{F}_\Gamma = \mathcal{F}_{\Gamma_0,K}$.  
In general, the $\Gal_K$-action on $\Gamma$ trivializes over a minimal finite Galois extension $R/K$, called the {\it resolvent field} of $\Gamma$. As the $\Gal_K$-action on $G_{F/K}$ is certainly trivialized over the Galois closure $\tilde F$, we have $R \subset \tilde F$. To keep things simple, we assume from now on that $(m, [R \colon K]) = 1$.
\begin{lemma}
    Let $H = \Gal(R/K)$. Then $ G_{F/K} \simeq \Gamma \rtimes \underline{H}$.  
\end{lemma}
\begin{proof}
    The assumption $(m, [R \colon K]) = 1$, implies that $L:= F \otimes_K R$ is a field extension of $R$. Since $\Gamma(R) = \Gamma(\overline{K})$ has order $m = [F\otimes_K R \colon R] = [L \colon R]$, we see that $L/R$ is Galois of degree $m$. It follows that $L/K$ is Galois, and since $L$ is contained in $\tilde F$, we must have $L = \tilde F$. By the Schur-Zassenhous theorem, we have $G_{L/K} \simeq \Gamma \rtimes G_{R/K} = \Gamma \rtimes \underline{H}$. 
\end{proof}

 Conversely, suppose $ \iota \colon G \hookrightarrow S_m$ is a transitive subgroup of $S_m$ with a normal subgroup $G_0 \subset G$ of order $m$, and suppose that $H := G/G_0$ has order prime to $m$. Then $G \simeq G_0 \rtimes H$, and any extension $F/K$ of degree $m$ with Galois group $\iota$ can be given the structure of $\Gamma$-extension, for some $K$-twist $\Gamma$ of $\underline{G_0}$. The resolvent field of $\Gamma$ is $\tilde F^{G_0}/K$. Thus, $\mathcal{F}_\Gamma$ is the family of $G$-extensions with resolvent field $R$, where $G:= \Gamma_0 \rtimes H \stackrel{\iota}{\hookrightarrow} S_m$. 

In this section, we elaborate on the heuristics presented in the introduction for the distribution of $\Cl_{F/K}[2]$ for $F \in \mathcal{F}_\Gamma$.  The family $\mathcal{F}_\Gamma$ sits in the larger family $\mathcal{F}_{\iota,K}$ of degree $m$ extensions $F/K$ with Galois group $G \stackrel{\iota}{\hookrightarrow} S_m$, and it is natural to guess that the distribution of $\Cl_{F/K}[2]$ for $F \in \mathcal{F}_\Gamma$ is the same as for $F \in \mathcal{F}_{\iota, K}$. The latter distribution has been predicted by Malle, at least under certain  irreducibility assumptions \cite{Malle2010, AdamMalle}; as discussed in Section \ref{sec: random model for Cl[2]}. Sawin--Wood give a prediction without such assumptions, but assuming $|G|$ is odd.

\subsection{Case $|G|$ odd}

If $\Gamma$ has trivial $\Gal_K$-action, i.e.\ $\Gamma = \underline{G}$ for some group $G$, then $\mathcal{F}_\Gamma = \mathcal{F}_{\iota, K}$ and we simply follow the heuristics of Sawin--Wood. More generally, if $G =  \Gamma_0 \rtimes H$ has odd order, then our heuristics for $\mathcal{F}_\Gamma$ revert to the  heuristics of Sawin--Wood for the larger family $\mathcal{F}_{\iota, K}$ (see \cite[\S10]{SawinWood}). The reason for this is that Hecke reciprocity is not relevant when $|G|$ is odd, by the following result.

\begin{proposition}\label{prop: G odd}
    Suppose $|G| = |G_{F/K}|$ is odd. Then $F/K$ is Hecke unramified. 
\end{proposition}
\begin{proof}
    If $F/K$ is an odd degree {\it Galois} extension, then the formula \cite[p. 64]{Serre1979} for the valuation of $\Diff_{F/K}$ in terms of ramification groups shows that $\Diff_{F/K}$ is a square as an ideal. Thus, in the notation of \S\ref{subsec: Hecke primes local fields}, the integers $d_i$ are all even, hence  $d = \sum d_i f_i$ is even, and hence the exponents $h_i = de_i - d_i$ of $H_{F/K}$ are all even. Thus, $H_{F/K}$ is a square, i.e.\ $F/K$ is Hecke unramified and admits no Hecke primes. 
    
    Now consider the general case where $F/K$ is not necessarily Galois. Since $\tilde F/K$ is an odd degree Galois extension, all ramification indices in {\it all} subextensions of $\tilde F/K$ are odd. Now let $j \colon I_F \to I_{\tilde F}$ be the map on ideal groups induced by the inclusion $F \hookrightarrow \tilde{F}$. The formula $\Diff_{\tilde F/K} = \Diff_{\tilde F/F} \Diff_{F/K}$ shows that $j(\Diff_{F/K})$ is a square in $I_{\tilde F}$. It follows that $\Diff_{F/K}$ is a square already in $I_F$, because all ramification indices of $\tilde F/F$ are odd. Thus, $H_{F/K} = \Disc_{F/K}\Diff_{F/K}^{-1}$ is a square as well, and there are no Hecke primes in $F/K$.   
\end{proof}

\subsection{Case $|G|$ even}
To formulate our heuristic when $|G|$ is even, write $\mathcal{F}_\Gamma = \mathcal{F}_\Gamma^1 \coprod \mathcal{F}_\Gamma^2$, where $\mathcal{F}_\Gamma^1$ (resp.\ $\mathcal{F}_\Gamma^2$) is the subfamily of $\Gamma$-extensions $F/K$ that are Hecke unramified (resp.\ Hecke ramified).

\begin{proposition}\label{prop: hecke primes divide the resolvent disc}
    Let $F \in \mathcal{F}_\Gamma$ and suppose $w \mid v$ is a Hecke prime of $F$. Then either $v$ ramifies in the resolvent field $R/K$ or there is a prime $w_i$ over $v$ which is wildly ramified. 
\end{proposition}
\begin{proof}
Write $v = \prod_{i = 1}^N w_i^{e_i}$, for primes $w_i$ of $F$ and let $f_i$ be the inertia degree of $w_i$ over $v$. We may suppose that each $w_i$ is at most tamely ramified over $v$. We suppose $v$ is unramified in $R/K$ in order to obtain a contradiction. Since $R/K$ is unramified, and since $\tilde F$ is the compositum of $F$ and $R$, it follows that $\tilde F/F$ is unramified. Since $\tilde F/K$ is Galois, it follows that all $e_i$ are equal, say $e_i = e$. Write $f = \sum_i f_i$. Since $\sum_i e_if_i = e \sum_i f_i = ef = m$, we see that both $e$ and $f$ are odd. The assumption that each $w_i$ is tamely ramified means that $\Diff_{F \otimes K_v/K_v} = \prod_i w_i^{e-1}$. Thus $\Disc_{F \otimes K_v/K_v} = v^{(e-1)\sum_i f_i} = v^{(e-1)f} = v^{m-f}$. The local Hecke ideal is therefore 
\[H_{F\otimes K_v /K_v} = \prod_i \m_i^{e(m-f) - e + 1} = \prod_i \m_i^{em - m -e +1},\]
where $\m_i$ is the completion of $w_i$.
Since each exponent $em-m-e+1$ is even, this is the square of an ideal, contradicting the fact that there is a Hecke prime $w$ above $v$.  
\end{proof}

By Proposition \ref{prop: hecke primes divide the resolvent disc}, the partition of $\mathcal{F}_\Gamma = \mathcal{F}_\Gamma^1 \coprod \mathcal{F}_\Gamma^2$ amounts to imposing constraints on the possible $K_v$-algebras $F \otimes_K K_v$ for a {\it finite} set of primes $v$ of $K$, namely those $v$ dividing $(m!)\Disc_{R/K}$. In particular, the partition is defined by finitely many local conditions.

Assume now that $G \subset S_m$ is absolutely irreducible (as at the end of \S\ref{sec: random model for Cl[2]}).  Our heuristic is the following: for $F \in \mathcal{F}_\Gamma^2$ we predict Malle's distribution, i.e.\ for any finite $\F_2$-vector space $V$,
\begin{equation}\label{eq: malle moments II}
    \mathbb{E}(|\mathrm{Surj}(\Cl_{F/K}[2],V)|) = \dfrac{| \wedge^2 V|}{|V|^u},
\end{equation}
where $u$ is defined as in \cite{Malle2010}. For $F \in \mathcal{F}_\Gamma^1$ we predict the distribution
\begin{equation} \label{eqn: spin moments}
    \mathbb{E}(|\mathrm{Surj}(\Cl_{F/K}[2],V)|) = \dfrac{| \wedge^2 V|}{|V|^{u-1}},
\end{equation}
taking into account the vacuity of Hecke reciprocity in this subfamily and the random model of Section \ref{sec: random model for Cl[2]}. Thus, the distribution on  $\mathcal{F}_\Gamma$ is a linear combination of these two distributions, weighted by the densities of the two subfamilies. Of course, these densities may be sensitive to the way $\mathcal{F}_\Gamma$ is ordered, so it is better to view them as two different families.

As we show below, it often happens that $\mathcal{F}_\Gamma^1$ is empty and hence $\mathcal{F} = \mathcal{F}_\Gamma^2$. In this case our heuristic simply says that in $\mathcal{F}_\Gamma$ the group $\Cl_{F/K}[2]$ is distributed as in the larger family $\mathcal{F}_{\iota, K}$. Only when $\mathcal{F}_\Gamma^1$ is non-empty do the heuristics change, and in this case we say that $\Gamma$ is {\it aberrant}.  Sawin--Wood point out inconsistencies in Malle's conjecture when $G \subset S_m$ is not absolutely irreducible \cite[\S12.3.2]{SawinWood}, which is why we avoid that case. However, we only really use the rough shape of Malle's heuristic, the main point being that the parameter $u$ should be lowered by $1$ on the subfamily $\mathcal{F}_\Gamma^1$. In any case, the notion of $\Gamma$-being aberrant is well-defined even when $G$ is not absolutely irreducible.

\subsection{Examples} \label{subsec: Gamma-extensions heuristics examples}
We survey some examples of families $\mathcal{F}_\Gamma$ with $|G|$ even and $K = \Q$.    
\subsubsection{$G = C_3 \rtimes C_2 \simeq S_3$}
Here $\Gamma$ is a twist of $\underline{C_3}$ and $H =C_2$. The resolvent field $R/\Q$ is quadratic and $\mathcal{F}_\Gamma$ is the family of cubic fields with quadratic resolvent $R$. 

\begin{lemma}
    $\Gamma$ is aberrant if and only if $R = \Q(i),\Q(\sqrt{-3})$ or $\Q(\sqrt{3})$.
\end{lemma}
\begin{proof}
    For $F \in \mathcal{F}_\Gamma$, Lemma \ref{lem: cubic characterization} says that a prime $v$ is Hecke ramified in $F/\Q$ if and only if $v$ splits as $1^2 1$ and the ramified local quadratic extension is generated by the square-root of a uniformizer when $v = 2$. When $v \neq 3$, this happens if and only if $v$ ramifies in $R$ (and $\Disc_{R/\Q}$ is divisible by $8$ if $v = 2$). Indeed, if $v$ is totally ramified in $F$, then the Galois closure of $F \otimes \Q_v/\Q_v$ is obtained by adjoining third roots of unity, hence is unramified over $F_w$. It follows that $R \otimes \Q_v$ is unramified over $\Q_v$ as well.  We conclude that if $\Gamma$ is aberrant, then $R$ must be unramified away from $2$ and $3$, and $\Disc(R)$ must not be divisible by $8$. This cuts the possibilities down to the three listed $R$. 
    
    To finish the proof we must show that  $\mathcal{F}_\Gamma^1$ is non-empty in each of these three cases. For such a $\Gamma$, every $F \in \mathcal{F}_\Gamma$ is Hecke unramified away from $3$, so the question is what happens at $3$. For $R = \Q(i)$, there are no Hecke primes above $3$ (since if $F$ is ramified at $3$, it must be totally ramified and hence Hecke unramified by Lemma \ref{lem: locally unobstructed at unramified primes}).  Hence $\mathcal{F}_\Gamma^1 = \mathcal{F}$ in this case. For the other two cases, we use the fact that there exists a cubic field totally ramified at $3$ in each of these families, e.g. $\Q(\sqrt[3]{3})$ and $\Q(\theta)$, where $\theta$ is a root of $x^3 - 15x - 20$. 
\end{proof}

Let us analyze the three aberrant families in more detail. 
\begin{enumerate}
    \item 
When $R = \Q(i)$, the family $\mathcal{F}_\Gamma$ is given by $F_n = \Q[x]/(f_n)$, with $f_n = x^3 + nx^2 + 9x + n$. The proof above shows that $\mathcal{F}_\Gamma =\mathcal{F}_\Gamma^1$ (i.e.\ there are no Hecke primes in the whole family) and so we predict the distribution of $\Cl_F[2]$ for $F\in \mathcal{F}_\Gamma$ is $(\ref{eqn: spin moments})$.  See Table \ref{tbl: cubic field with Gaussian resolvent} for a comparison of this prediction with some data.
\item When $R = \Q(\sqrt{-3})$ or $\Q(\sqrt{3})$, there are two possibilities for the splitting type at $3$: $(1^3)$ or $(1^21)$. In the first case $F \in \mathcal{F}_\Gamma^1$ since there are no Hecke primes, and in the second case $F \in \mathcal{F}_\Gamma^2$, as there is a unique Hecke prime (above $3$). For $F \in \mathcal{F}_\Gamma^1$, we predict the distribution $(\ref{eqn: spin moments})$, while for $F \in \mathcal{F}_\Gamma^2$ we predict Malle's distribution $(\ref{eq: malle moments II})$. Theorem \ref{main thm: over Q} is evidence for this in the case $R = \Q(\sqrt{-3})$. For the case $R = \Q(\sqrt{3})$, which corresponds to the family of cubic polynomials $x^3 - 3nx^2 - 3x + n$, see Table \ref{tbl: cubic field with quadratic resolvent Q(sqrt(3))} for data backing up our prediction.

\end{enumerate}

If $R$ is not isomorphic to $\Q(i)$ or $\Q(\sqrt{\pm3})$, we expect the distribution $(\ref{eq: malle moments II})$ in the family $\mathcal{F}_\Gamma$.

\subsubsection{$G = D_p \subset S_p$}

One way to generalize the previous example is to consider for each prime $p \geq 3$ the dihedral group of order $2p$, and let $\mathcal{F}_{D_p}$ be the family of $D_p$-extensions $F/\Q$ of degree $p$.  Let $\Gamma$ be a $\Q$-group with underlying group $\Gamma_0 = C_p$ and suppose that $\Gal_K$ acts through a (non-trivial) quadratic character $\Gal_K \to C_2 \hookrightarrow \Aut(C_p) \simeq C_{p-1}$.  So again the resolvent field $R/\Q$ of $\Gamma$ is quadratic, and $\mathcal{F}_\Gamma \subset \mathcal{F}_{D_p}$ is the family of $D_p$-extensions $F/\Q$ with quadratic resolvent $R \subset \tilde F$. 

\begin{proposition}\label{prop: dihedral hecke}
    Suppose $v \notin \{2,p\}$ is a prime of $\Q$ which is ramified in $R/\Q$. Then for all $F \in \mathcal{F}_\Gamma$, there exists a Hecke prime $w \mid v$ in $F/\Q$. 
\end{proposition}

\begin{proof}
Since the inertia group of a prime $u$ above $v$ in $\tilde F$ is cyclic, $u$ has ramification index $2$, and the ramification indices $e_i$ of the primes $w_i$ of $F$ over $v$ are all either $1$ or $2$. At least one $e_i$ is greater than $1$, so the inertia degrees $f_i$, which a priori are either $1$ or $p$, must all equal $f_i = 1$. Thus the inertia group coincides with the decomposition group $D(u|v)$. The primes $w_i$ correspond to the orbits of $D(u|v)$ on $\{1,\ldots, p\}$. Since $D(u|v)$ is generated by a reflection with exactly one fixed point, there must be $N = (p-1)/2$ such primes. The splitting type of $v$ in $F$ is thus $(1 1^2 1^2 \ldots 1^2)$. 

Write $v = \prod w_i^{e_i}$ and $\Diff_{F/\Q} = \prod w_i^{d_i}$. The discriminant exponent $d = \sum d_i f_i$ is equal to $\sum d_i = (p-1)/2$.  The Hecke ideal is therefore $\prod w_i^{de_i - d_i} = \prod w_i^{e_i(p-1)/2 - d_i}$. The unramified prime $w_1$ has exponent $(p-1)/2$, which is odd if and only if $p \equiv 1\pmod{4}$, but even in that case $w_1$ fails the second technical condition in the definition of a Hecke prime. However, the $g := (p-1)/2$ ramified primes $w_2,\ldots, w_{g+1}$ have $e_i = 2$ and hence Hecke exponent $p-1 - 1 = p-2$, which is odd. These are indeed Hecke primes, as the technical condition is always satisfied by a prime such that $e_if_i$ is even (see the proof of Lemma \ref{lem: local selmer ratio}). 
\end{proof}

Thus, for $F \in \mathcal{F}_\Gamma$ to be Hecke unramified, $R$ must be unramified away from $\{2,p\}$. As in the cubic case, a more careful analysis at $2$ also rules out the case that $R \otimes \Q_2/\Q_2$ is ramified and generated by the square-root of a uniformizer. Thus, in degree $p$ there are three aberrant $\Gamma$, corresponding to $R = \Q(i)$, $\Q(\sqrt{p})$, or $\Q(\sqrt{-p})$, a direct generalization of the cubic case.  For other resolvent fields $R/\Q$, we expect the  distribution of $\Cl_F[2]$ for $F \in \mathcal{F}_\Gamma$ to match that for $F \in \mathcal{F}_{D_p}$. 

Note that $D_p \subset S_p$ is not absolutely irreducible in Malle's sense unless $p = 3$, so Malle's heuristics in this case are more involved and on shakier ground (as previously mentioned). However, see \cite{Malle2010}, who discusses the case $p = 5$ in some detail.

\subsubsection{$G \subset F_p \subset S_p$}

More generally, if $\Gamma$ has underlying group $C_p$, then there is some subgroup $H \subset C_{p-1} = \Aut(C_p)$ through which the $\Gal_K$-action on $\Gamma$ factors. Then $\mathcal{F}_\Gamma$ is a subfamily of $\mathcal{F}_{G, \Q}$, where $G = C_p \rtimes H \subset F_p \simeq C_p \rtimes C_{p-1}$ is a subgroup of the Frobenius group $F_p \subset S_p$ of order $p(p-1)$.  The resolvent field is a (cyclic) $H$-extension $R/\Q$.

If we fix $p$ and consider all twists $\Gamma$ of $\underline{C_p}$, then it is natural to guess, following Proposition \ref{prop: dihedral hecke}, that only finitely many such $\Gamma$ are aberrant. We saw above that this is true for $p = 3$ and it is true for all $p$ if we restrict to dihedral twists, but we do not know whether this is true in general. We simply note that one example of an aberrant $\Gamma$ with $H = C_{p-1}$ (i.e.\ $G = F_p$) is $\Gamma = \mu_p$. We study this example in more detail in the next section.

\subsubsection{Other examples}
The examples above cover all $\Gamma$-extensions of degree $|\Gamma|$ with $|G_\Gamma|$ even, up to  degree $7$. In degree $9$ we encounter other primitive transitive subgroups $G \subset S_9$, the smallest being  $C_3^2 \rtimes C_4$. 
%\artane{specify where labeled.} 
More generally, one can consider a group $\Gamma_0$ of odd order and let $G$ be a subgroup of $\Gamma_0 \rtimes \Aut(\Gamma_0)$ containing $\Gamma_0$.  It would be interesting to collect data for $\Cl_{F/K}[2]$ for more such $\Gamma$.

It would also be desirable to relax some of our assumptions, e.g.\ that $F/K$ has degree $|\Gamma|$, that $\Gamma$ has odd order, or that $(m,[R \colon K]) = 1$. For example, what can one say about the heuristics for $\Cl_{F/K}[2]$ in degree $n$ families of $S_n$-extensions with a fixed quadratic resolvent field (corresponding to $A_n \subset S_n$)? Proposition \ref{prop: hecke primes divide the resolvent disc} may fail in this level of generality, so that one of the subfamilies $\mathcal{F}_\Gamma^1$ or $\mathcal{F}_\Gamma^2$ may be a non-empty thin subfamily.

\section{Heuristics for class groups of Kummer extensions}\label{sec: kummer heuristics}
 Let $K = \Q$ and let $\Gamma = \mu_p$, for an odd prime $p > 2$.  In the notation of the previous section, the group $G_\Gamma$ is the Frobenius group of order $p(p-1)$. By Kummer theory,  $\mathcal{F}_{\mu_p}$ is the family of pure degree $p$ fields $F_n = \Q(\sqrt[p]{n})$. In this section, we make the predictions of \S\ref{sec: Gamma-extension heuristics} explicit for this family. 
 
 \begin{lemma}
     $\mu_p$ is an aberrant $\Q$-group. Moreover, $\mathcal{F}_{\mu_p}^1 \subset\mathcal{F}_{\mu_p}$ is the subfamily of wildly ramified fields $F_n$, and $\mathcal{F}_{\mu_p}^2 \subset\mathcal{F}_{\mu_p}$ is the subfamily of tamely ramified fields $F_n$.  
 \end{lemma}
 \begin{proof}
     Lemma \ref{lem: Kummer Hecke ramification} shows that $F_n$ is Hecke unramified if and only if $F_n$ is wildly ramified at $p$. When $n$ is a $p$-th power in $\Q_p$, the field $F_n$ is tamely ramified at $p$, which shows that $\mathcal{F}_\Gamma^{1}$ is non-empty and hence $\mu_p$ is aberrant. 
 \end{proof}

The Galois group $G_\Gamma = G_{F_n/\Q}$ is isomorphic to the Frobenius group $G_p$ of order $p(p-1)$, i.e.\ the  semidirect product of $\Z/p\Z$ with $(\Z/p\Z)^\times$, with respect to the natural scalar action. (We have changed notation from $F_p$ to $G_p$ to avoid confusion with the fields $F_n$.) The short exact sequence
\[0 \to \Z/p\Z \to G_p \to (\Z/p\Z)^\times \to 0\]
corresponds to 
\[0 \to \Gal(\tilde F_n/\Q(\zeta_p)) \to G_{F_n/\Q} \to \Gal(\Q(\zeta_p)/\Q) \to 0.\]
We record the representation theory of $G_p$ below. 
\begin{lemma}\label{lem: rep theory of G_p}
    The irreducible $\C$-representations of $G_p$ are the $p-1$ characters $G_p \to (\Z/p\Z)^\times$ and the $(p-1)$-dimensional representation $\mathrm{cyc}$, which has a model over $\Q$ given by the natural action of $\mu_p \rtimes \Gal(\Q(\zeta_p)/K)$ on $\Q(\zeta_p)$, with $\mu_p \simeq \Z/p\Z$ acting by multiplication.  
\end{lemma}

\begin{proof}
Note that $\mathrm{cyc}$ is irreducible as  $\mu_p \rtimes \Gal(\Q(\zeta_p)/\Q)$-representation, since $G_p$ permutes the $p-1$ distinct irreducible $\Gal(\Q(\zeta_p)/\Q)$-subrepresentations. Since $(p-1)^2 + \sum_{i = 1}^{p-1} 1^2 = p(p-1) = |G_p|$, these are all the irreducible representations of $G_p$. 
\end{proof}

Let $G' \subset G_p$ be the subgroup (cyclic of order $p-1$) fixing $F_n$. 
\begin{lemma}\label{lem: kummer is absolutely irreducible}
    The induced representation $\mathrm{Ind}_{G'}^{G_p} 1_{G'}$ is isomorphic to $1_{G_p} \oplus \mathrm{cyc}$. 

\end{lemma}
\begin{proof}
    The representation $\mathrm{Ind}_{G'}^{G_p} 1$ is a $p$-dimensional subrepresentation of the regular representation. By the description of the irreducible representations of $G_p \simeq G_\Gamma$ given in Lemma \ref{lem: rep theory of G_p}, it must be isomorphic to $\chi \oplus \mathrm{cyc}$ for some character $\chi$. By Frobenius reciprocity, we have $\chi = 1$. 
\end{proof}

Lemma \ref{lem: kummer is absolutely irreducible} shows that $F_p \subset S_p$ is absolutely irreducible. Our prediction for the family $\mathcal{F}^1_\Gamma$ of wild $F_n$ therefore states that for any finite dimensional $\F_2$-vector space $V$, 
\begin{equation} \label{eqn: Kummer wild moments}
    \mathbb{E}(|\mathrm{Surj}(\Cl_F[2],V)|) = \dfrac{| \wedge^2 V|}{|V|^{(p-3)/2}},
\end{equation}
whereas for the tame fields $F_n \in \mathcal{F}_\Gamma^2$ we predict 
\begin{equation}\label{eq: Kummer tame moments}
    \mathbb{E}(|\mathrm{Surj}(\Cl_F[2],V)|) = \dfrac{| \wedge^2 V|}{|V|^{(p-1)/2}}.
\end{equation}

In Appendix \ref{sec: appendix tables}, we have tabulated some numerical evidence for these predictions. In the following subsection, we explain how these predictions suggest Conjecture \ref{conj: probability that h/h3 = 1}.
\subsection{Conjecture \ref{conj: probability that h/h3 = 1}}

By \cite{ShanksWilliams}, for primes $n$, the class number of $\Q(\sqrt[3]{n})$ is divisible by $3$ if and only if $n \equiv 1\pmod{3}$, so we assume $n \equiv 2\pmod{3}$. Now, $\Q(\sqrt[3]{n})$ forms a Hecke ramified family as $n$ varies over primes congruent to $8 \pmod{9}$ and a Hecke unramified family over primes congruent to $2,5 \pmod{9}$. 

Assuming that the distributions of the different $\ell^\infty$-parts of the class group are independent of one another, the probability that $h = 1$ should be the product over all primes $q$ of the probabilities that $\Cl_{\Q(\sqrt[3]{n})}[q^\infty]$ is trivial. By (\ref{eqn: altered Malle}), (\ref{eqn: Malle again}), and Proposition \ref{prop: main distribution}, the probability that $\Cl_{\Q(\sqrt[3]{n})}[2^\infty]$ is trivial is given by $\prod_{j \ge 2} (1+2^{-j})^{-1}$ in the Hecke ramified case and to $\frac{2}{3}\prod_{j \ge 2}(1+2^{-j})^{-1}$ in the Hecke unramified case. For all other $q \neq 3$, both cases should follow the distribution given by Malle's conjecture giving a probability of $\Cl_{\Q(\sqrt[3]{n})}[q^\infty]$ being trivial of $\prod_{j \ge 1}(1+q^{-j})^{-1}$, see \cite[Lemma 12.5. and Proposition 7.1.]{SawinWood}. In this way, we recover Conjecture \ref{conj: probability that h/h3 = 1}. 

The same reasoning shows that we also expect  (\ref{eqn: class number 1 probability tame}) and (\ref{eqn: class number 1 probability wild}) to give the probability that $\Q(\sqrt[3]{n})$ has prime-to-$3$ part of the class number $1$ as $n$ varies over all primes $\equiv \pm 1 \pmod{9}$ and $\not\equiv \pm 1 \pmod{9}$ respectively, without any restriction on the congruence class of $n \pmod{3}$. See Table \ref{tbl: probability of h/h3 equal to 1} for data backing up this prediction. 

\subsection{A more general conjecture using the heuristics of \cite{SiadVenkatesh}} \label{subsec: spin conjecture derivation}

Given a conjecture on the distribution of the $q^\infty$-parts of the class group for $q \neq 3$, the reasoning above generalizes in a straightforward manner to imply a generalization of Conjecture \ref{conj: probability that h/h3 = 1} predicting the probability that the prime-to-$3$ part of the class number of $\Q(\sqrt[3]{n})$ should be any given integer $C$: replace the probabilities that the $q^\infty$-part of the class group is trivial by the probability that it has size equal to the $q$-part $C_q$ of the given integer $C$. 

For $q \neq 2,3$, this probability is given by summing the $s=1$, $\epsilon = 0$, $r=1$ case of \cite[Proposition 7.1.]{SawinWood} over all $q$-groups of size $C_q$, corresponding to Malle's distribution, and for $q=2$ by summing the $s=1$, $\epsilon = 1$, $r=1$ case of \textit{loc.\ cit.}, corresponding to \cite{SiadVenkatesh}'s distribution (\ref{eqn: SV distribution}). 

To give a concrete example, as $n$ varies over primes congruent to $8 \pmod{9}$, the probability that $\Q(\sqrt[3]{n})$ has class number $2$ should be $\frac{3}{8} \cdot
\prod_{q \neq 3} \prod_{j \ge 2} (1+q^{-j})^{-1} \approx 0.2123$. On the other hand, as $n$ varies over primes congruent to $2,5 \pmod{9}$, the probability should be $\frac{1}{2} \cdot \prod_{q \neq 3} \prod_{j \ge 2} (1+q^{-j})^{-1} \approx 0.2831$.\footnote{A field has class number $2$ if it has unique length of factorizations, but not unique factorization! See \cite{Carlitz}. } See Table \ref{tbl: probability of h/h3 equal to 2} for data supporting this prediction.

\appendix

\section{Tables} \label{sec: appendix tables}

The tables of class group data below were computed using \textsc{Magma} \cite{Magma}. The computations were run on Intel(R) Xeon(R) Silver $4216$ CPU $@$ $2.10$GHz cores of the server of the School of Mathematics of the Institute for Advanced Study. The longest computation took roughly $12$ hours. 
The code and raw output used to produce these tables may be found on \href{https://github.com/ajsiad/Hecke-reciprocity-and-class-groups}{GitHub}. To run the computations in \textsc{Magma}, the GRH bound was turned on with \texttt{SetClassGroupBounds("GRH")} before calling the \texttt{ClassGroup()} function. Table entries are rounded to four decimal places.

In the tables below, various $H$-moments for $H$ a finite abelian $2$-group are presented. The shaded cells are the predictions of Equation (\ref{eqn: altered Malle}) of \S \ref{subsec: Gamma extension heuristics intro} on the behavior of $\Cl_{K/F}[2]$ when there are no Hecke primes. The rest of the predicted values are those coming either from Equation (\ref{eqn: Malle again}) of \S \ref{subsec: Gamma extension heuristics intro} or from Equation \ref{eqn: SV distribution} of \S \ref{sec: connection to spin structures}. The agreement between the $H$-moments we observe and the values we predict is less strong as $\lvert H \rvert$ grows. However, the ratios of the observed $H$-moments for type I and type II families tend to stay relatively close to the ratio we predict. We believe this likely arises from second order terms in the asymptotic expansions of the $H$-moments for both families which are comparable.

In the tables on pure prime degree extensions, the headings type I and type II refer to extensions wild and tame respectively at the prime degree, in accordance with Dedekind's division \cite{Dedekind1900}.

\subsection{Pure cubic fields} \label{tbl: pure cubic fields}

We sampled the $2$-primary part of the class group of $\Q(\sqrt[3]{n})$, for $n \geq 10^7$ and cube-free.  The sample consisted of the first $100000$ fields found, of which $76926$ were of type I and $23074$ of type II. The average discriminant was $\approx -10^{15}$.  \\

\begin{center}
\begin{tabular}{|c|c||c|c||c|c|}\hline
\multirow{2}{*}{\textbf{Moment}} & \multicolumn{2}{c|}{\textbf{Observed}} & \multicolumn{2}{c|}{\textbf{Predicted}} & \multirow{2}{*}{\textbf{Malle}} \\ \cline{2-5}
& \textit{type I} & \textit{type II} & \textit{type I} & \textit{type II} &  \\ \hline 
$\Z/2$ & 0.9118 & 0.4288 & \cellcolor{lightgray} 1 &  1/2 & 1/2 \\ \hline 
$\Z/4$ & 0.4589 & 0.2138 & 1/2 & 1/4 & 1/4 \\ \hline
$\Z/8$ & 0.2339 & 0.1037 & 1/4 & 1/8 & 1/8 \\ \hline
$\Z/2 \times \Z/2$ & 1.5214 & 0.3308 & \cellcolor{lightgray} 2 & 1/2 & 1/2 \\ \hline
$\Z/4 \times \Z/2$ & 0.8089 & 0.1619 & 1  & 1/4 & 1/4 \\ \hline
$\Z/4 \times \Z/4$ & 0.4942 & 0.0374 & 1/2 & 1/8 & 1/8 \\ \hline
$\Z/2 \times \Z/2 \times \Z/2$ & 4.4683 & 0.5097 & \cellcolor{lightgray} 8 &  1 & 1 \\ \hline
\end{tabular}
\end{center}

\subsection{Probability of prime-to-3 class number one in $\Q(\sqrt[3]{p})$} 
\label{tbl: probability of h/h3 equal to 1}

Let $h$ and $h_3$ denote the class number and the $3$-part of the class number respectively. We sampled the probability that $h/h_3 = 1$ for cubic fields  $\Q(\sqrt[3]{p})$, with $p$ prime. The sample consisted of the first $120000$ primes $p > 10^7$, split according to the congruence class of $p \pmod{9}$. The average discriminant of the sample was $\approx -10^{15}$. Each congruence class contained roughly $20000$ fields.

Contrast with the tables of \cite{ShanksWilliams} and \cite{TennenhouseWilliams}. The shaded cells come from Conjecture \ref{conj: probability that h/h3 = 1}. \\

\begin{center}
\begin{tabular}{|c|c|c|c|c|c|c|}\hline
\multirow{2}{*}{$\P(h/h_3 = 1)$} & \multicolumn{6}{c|}{\textbf{Congruence classes}} \\ \cline{2-7}
 & \textbf{1} & \textbf{2} & \textbf{4} & \textbf{5} & \textbf{7} & \textbf{8} \\ \hline
\textbf{Predicted} & 0.5662 & \cellcolor{lightgray} 0.3775  & 
0.3775 & \cellcolor{lightgray} 0.3775 &  0.3775 & \cellcolor{lightgray} 0.5662 \\ \hline 
\textbf{Observed} & 0.5863  & 0.3889 & 0.3826 & 0.3917 & 0.3943 & 0.5863 \\ \hline 
\end{tabular}
\end{center}

\subsection{Probability of prime-to-3 class number two in $\Q(\sqrt[3]{p})$} 
\label{tbl: probability of h/h3 equal to 2}

For the same sample as Table \ref{tbl: probability of h/h3 equal to 1}, the observed probability that $h/h_3 = 2$ is below. The predicted values come from \S \ref{sec: kummer heuristics}. \\

\begin{center}
\begin{tabular}{|c|c|c|c|c|c|c|}\hline
\multirow{2}{*}{$\P(h/h_3 = 2)$} & \multicolumn{6}{c|}{\textbf{Congruence classes}} \\ \cline{2-7}
 & \textbf{1} & \textbf{2} & \textbf{4} & \textbf{5} & \textbf{7} & \textbf{8} \\ \hline
\textbf{Predicted} & 0.2123 & 0.2831 & 0.2831 & 0.2831 & 0.2831 & 0.2123 \\ \hline 
\textbf{Observed} & 0.2026 & 0.2863 & 0.2833 & 0.2902 & 0.2826 & 0.2028 \\ \hline 
\end{tabular}
\end{center}

\subsection{Cubic fields with quadratic resolvent $\Q(i)$} \label{tbl: cubic field with Gaussian resolvent} We sampled the $2$-primary part of the class group of cubic fields with quadratic resolvent field equal to the Gaussian rationals $\Q(i)$. These fields correspond to the two-parameter family of cubic polynomials $f(x) = cx^3+bx^2+9cx+b$ for $b,c \in \Z$, see \S \ref{subsec: Gamma-extensions heuristics examples}. The sample consisted of the first $50000$ fields found iterating over $b$ in the range $[10^4,10^4+100]$ and $c$ in the range $[10^4,10^4+5 \cdot 10^3]$. The average discriminant was $\approx -10^{19}$. \\

\begin{center}
\begin{tabular}{|c|c|c|c|}\hline
\textbf{Moment} & \textbf{Observed} & \textbf{Predicted} & \textbf{Malle} \\ \hline
$\Z/2$ & 1.0614 & \cellcolor{lightgray} 1 & 1/2 \\ \hline 
$\Z/4$ & 0.5625 & 1/2  & 1/4 \\ \hline
$\Z/8$ & 0.2706 & 1/4 & 1/8 \\ \hline
$\Z/2 \times \Z/2$ & 2.4223  & \cellcolor{lightgray} 2 & 1/2 \\ \hline
$\Z/4 \times \Z/2$ & 1.5275 & 1  & 1/4 \\ \hline
$\Z/4 \times \Z/4$ & 1.5110 & 1/2 & 1/8 \\ \hline
$\Z/2 \times \Z/2 \times \Z/2$ & 15.4123 & \cellcolor{lightgray} 8 & 1 \\ \hline
\end{tabular}
\end{center}

\subsection{Cubic fields with quadratic resolvent $\Q(\sqrt{3})$} \label{tbl: cubic field with quadratic resolvent Q(sqrt(3))} We sampled the $2$-primary part of the class group of cubic fields with quadratic resolvent field $\Q(\sqrt{3})$. These fields correspond to the one-parameter family of cubic polynomials $f(x) = x^3 - 3ax^2-3x+a$ for $a \in \Z$, see \S \ref{subsec: Gamma-extensions heuristics examples}. The sample consisted of the first $50000$ fields found with $a$ in the range $[100,10^5]$, of which $33334$ were of type I and $16666$ of type II. The average discriminant was $\approx 10^{19}$. \\

\begin{center}
\begin{tabular}{|c|c||c|c||c|c|}\hline
\multirow{2}{*}{\textbf{Moment}} & \multicolumn{2}{c|}{\textbf{Observed}} & \multicolumn{2}{c|}{\textbf{Predicted}} & \multirow{2}{*}{\textbf{Malle}} \\ \cline{2-5}
& \textit{type I} & \textit{type II} & \textit{type I} & \textit{type II} &  \\ \hline
$\Z/2$ & 0.7063 & 0.3318 & \cellcolor{lightgray} 1/2 &  1/4 & 1/4 \\ \hline 
$\Z/4$ & 0.1696 & 0.0894 & 1/8 & 1/16 & 1/16 \\ \hline
$\Z/8$ & 0.0437 & 0.0254 & 1/32 & 1/64 & 1/64 \\ \hline
$\Z/2 \times \Z/2$ & 1.0632 & 0.2264 & \cellcolor{lightgray} 1/2 & 1/8 & 1/8 \\ \hline
$\Z/4 \times \Z/2$ & 0.2650 & 0.0658 & 1/8  & 1/32 & 1/32 \\ \hline
$\Z/4 \times \Z/4$ & 0.0403 & 0.0173 & 1/32 & 1/128 & 1/128 \\ \hline
$\Z/2 \times \Z/2 \times \Z/2$ & 3.0441 & 0.2722 & \cellcolor{lightgray} 1 &  1/8 & 1/8 \\ \hline
\end{tabular}
\end{center}

\subsection{Pure quintic fields} 
\label{tbl: pure quintic fields}

We sampled the $2$-primary part of the class group of pure quintic extensions $\Q(\sqrt[5]{n})/\Q$, for $n> 10^6$ fifth power free. The sample consisted of the first $100000$ fields found, of which $83995$ were of type I and $16005$ of type II. The average discriminant was $\approx 10^{27}$.  \\

\begin{center}
\begin{tabular}{|c|c||c|c||c|c|}\hline
\multirow{2}{*}{\textbf{Moment}} & \multicolumn{2}{c|}{\textbf{Observed}} & \multicolumn{2}{c|}{\textbf{Predicted}} & \multirow{2}{*}{\textbf{Malle}} \\ \cline{2-5}
& \textit{type I} & \textit{type II} & \textit{type I} & \textit{type II} &  \\ \hline 
$\Z/2$ & 0.5360 & 0.2789 & \cellcolor{lightgray} 1/2 &  1/4 & 1/4 \\ \hline 
$\Z/4$ & 0.1395 & 0.0699 & 1/8  & 1/16 & 1/16 \\ \hline
$\Z/8$ & 0.0379 & 0.0202 & 1/32 & 1/64 & 1/64 \\ \hline
$\Z/2 \times \Z/2$ & 0.6300 & 0.1927 & \cellcolor{lightgray} 1/2  & 1/8 & 1/8 \\ \hline
$\Z/4 \times \Z/2$ & 0.1831 & 0.0485 & 1/8 & 1/32 & 1/32 \\ \hline
$\Z/4 \times \Z/4$ & 0.0469 & 0.0000 & 1/32 & 1/128 & 1/128 \\ \hline
$\Z/2 \times \Z/2 \times \Z/2$ & 1.8241 & 0.3674 & \cellcolor{lightgray} 1 & 1/8  & 1/8 \\ \hline
\end{tabular}
\end{center}

\subsection{Pure septic fields} 
\label{tbl: pure septic fields}

We sampled the $2$-primary part of the class group of pure cubic extensions $\Q(\sqrt[7]{n})/\Q$, for $n > 10^7$ seventh power free. The sample consisted of the first $50000$ fields found, of which $43879$ were of type I and $6121$ of type II. The average discriminant was $\approx -10^{33}$. \\

\begin{center}
\begin{tabular}{|c|c||c|c||c|c|}\hline
\multirow{2}{*}{\textbf{Moment}} & \multicolumn{2}{c|}{\textbf{Observed}} & \multicolumn{2}{c|}{\textbf{Predicted}} & \multirow{2}{*}{\textbf{Malle}} \\ \cline{2-5}
& \textit{type I} & \textit{type II} & \textit{type I} & \textit{type II} &  \\ \hline 
$\Z/2$ & 0.2758 & 0.1505 & \cellcolor{lightgray} 1/4 & 1/8 & 1/8 \\ \hline 
$\Z/4$ & 0.0347 & 0.0186 & 1/32 & 1/64 & 1/64 \\ \hline
$\Z/8$ & 0.0055 & 0.0020 & 1/256 & 1/512 & 1/512 \\ \hline
$\Z/2 \times \Z/2$ & 0.2133 & 0.0480 & \cellcolor{lightgray} 1/8 & 1/32 & 1/32 \\ \hline
$\Z/4 \times \Z/2$ & 0.0202 & 0.0039 & 1/64 & 1/256 & 1/256 \\ \hline
$\Z/4 \times \Z/4$ & 0.0044 &  0.0000 & 1/512 & 1/2048 & 1/2048 \\ \hline
$\Z/2 \times \Z/2 \times \Z/2$ & 1.0491 & 0.0274 & \cellcolor{lightgray} 1/8 & 1/64 & 1/64 \\ \hline
\end{tabular}
\end{center}

\subsection{Pure cubic extensions of $\Q(\sqrt{3})$} \label{tbl: pure cubic extensions of Q(sqrt(3))}

We sampled the $2$-primary part of the relative class group of pure cubic extensions $K(\sqrt[3]{n})/\Q(\sqrt{3})$, for $(n)$ a cubefree ideal with $n \in \O_{\Q(\sqrt{3})}$ of the form $a+b\sqrt{3}$, $a,b \in \Z$, with $a$ in the range $[10^4,10^4+500]$ and $b \in [10^4+10^3,10^4+10^3+300]$. The sample consisted of the first $100000$ fields found, of which $92321$ were of type I and $7679$ of type II. The average absolute discriminant was $\approx 10^{23}$.  \\

\begin{center}
\begin{tabular}{|c|c||c|c||c|c|}\hline
\multirow{2}{*}{\textbf{Moment}} & \multicolumn{2}{c|}{\textbf{Observed}} & \multicolumn{2}{c|}{\textbf{Predicted}} & \multirow{2}{*}{\textbf{Malle}} \\ \cline{2-5}
& \textit{type I} & \textit{type II} & \textit{type I} & \textit{type II} &  \\ \hline 
$\Z/2$ & 0.4717 & 0.4876 & \cellcolor{lightgray} 1/2 & \cellcolor{lightgray} 1/2 & 1/4 \\ \hline 
$\Z/4$ & 0.2378 & 0.2469 & 1/8 & 1/8 & 1/16 \\ \hline
$\Z/8$ & 0.0623 & 0.0672 & 1/32 & 1/32 & 1/64 \\ \hline
$\Z/2 \times \Z/2$ & 0.4245 & 0.4344 & \cellcolor{lightgray} 1/2 & \cellcolor{lightgray} 1/2 & 1/8 \\ \hline
$\Z/4 \times \Z/2$ & 0.2125 & 0.2084 & 1/8 & 1/8 & 1/32 \\ \hline
$\Z/4 \times \Z/4$ & 0.1955 & 0.2000 & 1/32 & 1/32 & 1/128 \\ \hline
$\Z/2 \times \Z/2 \times \Z/2$ & 0.7188 & 0.5251 & \cellcolor{lightgray} 1 & \cellcolor{lightgray} 1 & 1/8 \\ \hline
\end{tabular}
\end{center}

\subsection{Pure cubic extensions of $\Q(\sqrt{-2})$} 
\label{tbl: pure cubic extensions of Q(sqrt(-2))}

We sampled the $2$-primary part of the relative class group of pure cubic extensions $K(\sqrt[3]{n})/\Q(\sqrt{-2})$, for $(n)$ a cubefree ideal with $\Nm(n)$ in the range $[10^7+3 \cdot 10^3,10^7+60 \cdot 10^3]$. The letters W, T, and their order, indicate the ramification type of the two primes of $\Q(\sqrt-2)$ above $3$ in $K(\sqrt[3]{n})$. The sample consisted of the first $100000$ fields found, of which $59231$ were of type WW, $17740$ of type WT, $17740$ of type TW, and $5289$ of type TT. The average absolute discriminant was $\approx -10^{19}$.  \\

\begin{adjustbox}{center}
\begin{tabular}{|c|c||c||c||c|c||c||c|c|}\hline
\multirow{2}{*}{\textbf{Moment}} & \multicolumn{4}{c|}{\textbf{Observed}} & \multicolumn{3}{c|}{\textbf{Predicted}} & \multirow{2}{*}{\textbf{Malle}} \\ \cline{2-8}
& \textit{WW} & \textit{WT} & \textit{TW} & \textit{TT} & \textit{WW} & \textit{WT/TW} & \textit{TT} & \\ \hline
$\Z/2$ & 0.4492 & 0.2169 & 0.2168 & 0.1993 & \cellcolor{lightgray} 1/2 &  1/4 &  1/4 & 1/4 \\ \hline 
$\Z/4$ & 0.1120 & 0.0545 & 0.0545 & 0.0393 & 1/8 & 1/16 & 1/16 & 1/16 \\ \hline
$\Z/8$ & 0.0272 & 0.0079 & 0.0079 & 0.0061 & 1/32 & 1/64 & 1/64 & 1/64 \\ \hline
$\Z/2 \times \Z/2$ & 0.3534 & 0.1015 & 0.1011 & 0.0545 & \cellcolor{lightgray} 1/2 &   1/8 &  1/8 & 1/8 \\ \hline
$\Z/4 \times \Z/2$ & 0.0948 & 0.0257 & 0.0257 & 0.0182 & 1/8 & 1/32 & 1/32 & 1/32 \\ \hline
$\Z/4 \times \Z/4$ & 0.0162 & 0.0108 & 0.0108 & 0.0000 & 1/32 & 1/128 & 1/128 & 1/128 \\ \hline
$\Z/2 \times \Z/2 \times \Z/2$ & 0.4680 & 0.0758 & 0.0758 & 0.0000 & \cellcolor{lightgray} 1 &  1/8 &
 1/8 & 1/8 \\ \hline
\end{tabular}
\end{adjustbox}

 \bibliographystyle{amsalpha}
 \bibliography{references}

\end{document}